\documentclass{lms}

\usepackage{amssymb}
\usepackage{epsfig}

\providecommand{\dist}{\mathop{\rm dist}\nolimits}

\newtheorem{theorem}{Theorem}[section] 
\newtheorem{lemma}[theorem]{Lemma}     

\newtheorem{theorem A}[theorem_l]{Theorem A}
\newtheorem{theorem B}[theorem_l]{Theorem B}
\newtheorem{theorem C}[theorem_l]{Theorem C}
\newtheorem{theorem D}[theorem_l]{Theorem D}
\newtheorem{theorem E}[theorem_l]{Theorem E}
\newtheorem{theorem F}[theorem_l]{Theorem F}
\newtheorem{theorem G}[theorem_l]{Theorem G}

\newtheorem{lemma A}{Lemma 6.1}

\newnumbered{assertion}{Assertion}    
\newnumbered{conjecture}{Conjecture}  
\newnumbered{definition}{Definition}
\newnumbered{hypothesis}{Hypothesis}
\newnumbered{remark}{Remark}
\newnumbered{note}{Note}
\newnumbered{observation}{Observation}
\newnumbered{problem}{Problem}
\newnumbered{question}{Question}
\newnumbered{algorithm}{Algorithm}
\newnumbered{example}[theorem]{Example}
\newunnumbered{notation}{Notation} 

\title[Euclidean and Hyperbolic Lengths of Images of Arcs]
 {Euclidean and Hyperbolic Lengths of Images of Arcs} %

\author{A. F. Beardon \and T. K. Carne}

\classno{30F45, 30C45, 30C25}

\begin{document}
\maketitle

\begin{abstract}
Let $f$ be a function that is analytic in the unit disc. We give new
estimates, and new proofs of existing estimates, of the Euclidean length of
the image under $f$ of a radial segment in the unit disc. Our methods are
based on the hyperbolic geometry of plane domains, and we address some new
questions that follow naturally from this approach.
\end{abstract}

\section{Introduction}\label{Introduction}

\noindent Let $f$ be a function that is analytic in the open unit disc
$\mathbb{D}$ in the complex plane $\mathbb{C}$. We are interested in obtaining
upper bounds of the Euclidean length
\[
{\cal E}(r,\theta) = \int_0^r|f'(te^{i\theta})|\,dt 
\]
of the $f\!$-image of the ray $[0,re^{i\theta}]$, and in understanding the
geometry that lies behind these bounds.  When we wish to emphasize the role of
$f$ in this expression (for example, when we are discussing several functions
at once) we shall use ${\cal E}_f(r,\theta)$ in the natural way.

The Dirichlet space ${\cal D}$ is the space of analytic functions
$f:\mathbb{D}\to\mathbb{C}$ with
\begin{equation}
A(f) = \int\int_\mathbb{D}|f'(z)|^2\,dxdy 
< +\infty.
\end{equation}
The quantity $A(f)$ is the area of the image $f(\mathbb{D})$, counting
multiplicity.  It is easy to see that if $f\in {\cal D}$ then ${\cal
E}(1,\theta)$ is finite for almost all $\theta$. Indeed, by the Cauchy-Schwarz
inequality,
\[
\left(\int_{1/2}^1|f'(te^{i\theta})|\,dt \right)^2 
\ \leqslant\
\int_{1/2}^1|f'(te^{i\theta})|^2\, t\, dt \; 
\int_{1/2}^1\,{dt\over t},
\]
so that
\[
\int_0^{2\pi}\left(\int_{1/2}^1|f'(te^{i\theta})|\,dt \right)^2\,d\theta 
\ \leqslant\  A(f)\,\log 2.
\]
It follows that if $f\in {\cal D}$, then ${\cal E}(1,\theta)$ is finite for
almost all $\theta$. Beurling (\cite{Beurling}, and \cite[p.\,344]{MZ}) has
proved the stronger result that if $f\in {\cal D}$ then ${\cal E}(1,\theta)$
is finite except when $e^{i\theta}$ lies in some subset of
$\partial\mathbb{D}$ of logarithmic capacity zero; thus ${\cal E}(r,\theta)
\to +\infty$ as $r\to 1$ for only a small set of $\theta$. The following
result gives an upper bound on ${\cal E}(r,\theta)$ for all $\theta$, and it
is the starting point of the work in this paper.

\begin{theorem A*}
\label{A}
Suppose that $f$ is in ${\cal D}$. Then, for each $\theta$,
\begin{equation}
{\cal E}(r,\theta) = 
o\left(\left[\log{1\over 1-r}\right]^{1/2}\right)
\end{equation}
as $r\to 1$. The exponent $1/2$ is best possible even for the subclass of
functions that are bounded and univalent in $\mathbb{D}$.
\end{theorem A*}

Theorem A is due to Keogh \cite{Keogh} (who attributes it to J.E.  Littlewood)
and also to Rosenblatt (see \cite[p.\,45]{Jenkins}).  Both Keogh and
Rosenblatt state the result for bounded univalent functions but, as remarked
by Jenkins \cite{Jenkins}, their proof is valid without change for functions
in ${\cal D}$.  Theorem A has been extended in \cite{CT} where the following
two local versions are proved (see \cite{CT}, pp.\,492-493 and Theorem 1).

\begin{theorem B*}\label{B}
Suppose that $f$ is analytic in $\mathbb{D}$, and that for some Stolz region
$S$ at $e^{i\theta}$, $f(S)$ has finite area.  Then {\rm (1.2)} holds for this
$\theta$.
\end{theorem B*}

\begin{theorem C*}\label{C}
There is a constant $A$ such that, if $f$ is analytic and univalent in
$\mathbb{D}$ with $f(0) = 0$, then
\[
{\cal E}(r,\theta) 
\ \leqslant\  
A \left(\log{1\over 1-r}\right)^{1/2}
\sup\{|f(te^{i\theta})| : 0\leqslant t\leqslant r\}
\]
for each $\theta$ and each $r\in (\frac{1}{2},1)$.  Further, if $f$ is bounded
on $[0,e^{i\theta})$ then {\rm (1.2)} holds for this $\theta$.
\end{theorem C*}

Theorem B appears to be stronger than Theorem A, and its proof in \cite{CT} is
substantially harder than a proof of Theorem A. We shall show that Theorems A
and B are equivalent up to a very simple argument in hyperbolic geometry
(which involves no function theory at all). Indeed, our first objective is to
understand Theorems A, B and C from a geometric point of view, and we shall
show that all three results are more transparent when placed in the context of
the hyperbolic geometry of a plane domain.

The class of universal covering maps includes the class of univalent maps, and
we know of no bounds in the literature on ${\cal E}(r,e^{i\theta})$ for
universal covering maps. The following simple example shows what can happen
for covering maps, and also illustrates the use of hyperbolic geometry in this
topic.

\begin{example}
The map $\varphi(z) = (z-i)/(z+i)$ is a conformal map of the upper half-plane
$\mathbb{H}$ onto $\mathbb{D}$, and $F(z) = \exp \big(i\log (-iz)\big)$ is a
universal covering map of $\mathbb{H}$ onto the annulus $A$ given by
$e^{-\pi/2} < |z| < e^{\pi /2}$. Now let $f = F\circ \varphi^{-1}$. Then $f$
is a universal covering map of $\mathbb{D}$ onto $A$, so that the hyperbolic
length of $f([0,r])$ is $\log\,(1+r)/(1-r)$. However, the hyperbolic metric in
$A$ is $\lambda_A(w)\, |dw|$, where
\[
\lambda_A(F(z))|F'(z)| 
\ =\  \lambda_\mathbb{H}(z) 
\ =\  {1\over {\rm Im}(z)},
\]
and a calculation shows that $\lambda_A(w)=1$ when $|w|=1$.  It follows that
the Euclidean length of $f([0,r])$ is the same as its hyperbolic length; thus
\begin{equation}
{\cal E}(r,0)
\ =\ \log {1+r\over 1-r} 
\ =\  \log {1\over 1-r}+\log 2+o(1)
\end{equation}
as $r\to 1$. The striking feature of this example is that we have `lost' the
exponent $1/2$ from the estimate in {\rm (1.2)}.
\end{example}

The exponent $1/2$ in Theorem A is best possible, and the work in this paper
was partly motivated by the desire to understand why this exponent appears in
the estimates for functions in the Dirichlet class but not for universal
covering maps. The explanation of this will be given in terms of hyperbolic
geometry in Section 9. We shall also discuss why, for geometric reasons, other
exponents arise in other circumstances. We regard this as the most interesting
part of the paper, and our geometric treatment of Theorems A, B and C should
be regarded as preparation for this work.

We come now to the idea that underpins most of the work in this paper. As the
hyperbolic metric (and not the Euclidean metric) is the natural metric on
$\mathbb{D}$, when we consider an analytic map $f:\mathbb{D}\to\mathbb{C}$ we
should primarily be concerned with {\it the change of scale from the
hyperbolic metric to the Euclidean metric}, and not with the Euclidean change
of scale $|f'(z)|$. Now the hyperbolic metric $\rho_\mathbb{D}$ on
$\mathbb{D}$ is given by $ds = \lambda_\mathbb{D}(z)|dz|$, where
$\lambda_\mathbb{D}(z) = 2/(1-|z|^2)$. It follows, then, that {\it the primary
role in this discussion should be given to the function}
\begin{equation}
\Lambda(f,z) = {|f'(z)|\over \lambda_\mathbb{D} (z)},
\end{equation}
and that we should consider ${\cal E}(r,\theta)$ to be the integral of this
function with respect to the hyperbolic length on $[0,1)$; that is,
\begin{equation}
{\cal E}(r,\theta) = \int_0^r
{|f'(te^{i\theta})|\over \lambda_\mathbb{D}(te^{i\theta})}\, 
\lambda_\mathbb{D}(te^{i\theta})\,dt.
\end{equation}
Further, as the geodesic segment $[0,re^{i\theta}]$ has hyperbolic length
$\ell(r)$, where
\[
\ell(r) = \rho_\mathbb{D}(0,r) = \log {1+r\over 1-r},
\]
it seems clear that (1.2) should be written as
\[
{\cal E}(r,\theta) = 
o\left(\ell(r)^{1/2}\right).
\]
The use of the apparently simpler term $\log (1-r)^{-1}$ in (1.2) only serves
to conceal the geometry behind these estimates. Of even greater importance is
the fact that the hyperbolic metric is conformally invariant whereas the
logarithmic term in (1.2) is not.  To summarize these ideas, our view is that
\[
{\cal E}(r, \theta) 
\ =\ \int_0^r \Lambda(f,t) \; d\rho_\mathbb{D}(t) 
\ =\  o(\left(\ell(r)^{1/2}\right).
\]

To develop this idea further, the class ${\cal D}$ should now be considered to
be the class of analytic maps $f:\mathbb{D}\to\mathbb{C}$ for which the basic
scaling function $|f'|/\lambda_\mathbb{D}$ is square-integrable over the
hyperbolic plane. Of course, one can also consider the other $L^p\!$-spaces of
functions $f$ for which
\[
\int\int_\mathbb{D} 
\left({|f'(z)|\over \lambda_\mathbb{D}(z)}\right)^p
\lambda_\mathbb{D}(z)^2\,dxdy 
\ <\  +\infty.
\]
These spaces occur in the theory of automorphic functions (they are the
$A^p_q$ spaces, with $q=2$, introduced by Bers), and we shall consider them
briefly in this context in Section 5. To illustrate these ideas, we pause to
show that {\sl if $f/\lambda_\mathbb{D}$ is square-integrable over the
hyperbolic plane (that is, if $f\in {\cal D}$), then
$f:\mathbb{D}\to\mathbb{C}$ is a Lipschitz map with respect to the natural
metrics on $\mathbb{D}$ and $\mathbb{C}$}. This is our next result, and we
give the best Lipschitz constant.

\begin{theorem}
Suppose that $f\in {\cal D}$. Then, for any $z$ in $\mathbb{D}$, 
\begin{equation}
{|f'(z)|\over \lambda_\mathbb{D} (z)} 
\ \leqslant\  
\sqrt{A(f)\over 4\pi}. 
\end{equation}
Further, for each $z$, equality occurs 
when $f(w) = (w-z)/(1-\bar z w)$.
\end{theorem}

This theorem shows that $|f'(z)|(1-|z|^2) = O(1)$, so $f$ is in the Bloch
class.  The proof, together with much more information about the Bloch class,
is given by Anderson, Clunie and Pommerenke in \cite{ACP}.

\begin{proof}
Suppose that $f(z) = \sum_n a_nz^n$, where $z\in\mathbb{D}$.  The
Cauchy-Schwarz inequality gives
\[
|f'(z)| 
\ \leqslant\  \sum_{n=1}^\infty
\left(\sqrt{n}|a_n|\Big)\Big(\sqrt{n}|z|^{n-1}\right)
\ \leqslant\  
\left(\sum n |a_n|^2\right)^{1/2}
\left(\sum n|z|^{2n-2}\right)^{1/2}.
\]
Since $A(f) = \pi \sum_{n=1}^\infty n|a_n|^2$ and
\[
\lambda_\mathbb{D}(z)^2/4 
\ =\  \left({1\over 1-|z|^2}\right)^2
\ =\  \sum_{n=1}^\infty n|z|^{2n-2}.
\]
This proves (1.6). The statement about equality is easily verified (note that
for these functions, $A(f)=\pi$).
\end{proof}

This simple application of the Cauchy-Schwarz inequality underlies all of the
results in this paper. The argument used in the proof above can be
strengthened slightly to show that
\begin{equation}
{|f'(z)| \over \lambda_\mathbb{D}(z)} \to 0 
\qquad \hbox{ as } \qquad 
z \to \partial \mathbb{D}.
\end{equation}
This will be clear from the proof of Theorem A in Section 3 (and it shows that
any $f$ in ${\cal D}$ is in the little Bloch class).

The plan of the paper is as follows. In Section 2 we briefly outline the main
ideas concerning the hyperbolic metric of a simple connected domain. Section 3
contains our discussion of Theorem A. This contains two examples to show that
the exponent $1/2$ in Theorem A is best possible, and both are based on
hyperbolic geometry. Although one involves a technical geometric estimate
(which is of value in its own right), these examples are completely
transparent. We also give a geometric proof of a result which includes Theorem
A for univalent functions as a special case. In Section 4 we show how Theorem
B follows from Theorem A in an elementary way using only hyperbolic
geometry. This proof is considerably shorter than the published proof of
Theorem B; it is based on a conformally invariant form of Theorem A, and it
applies without change to many other subdomains of $\mathbb{D}$.  Section 5 is
concerned with Theorem C and it contains a geometric proof of this and some
results about $L^p\!$-spaces. The proof of Theorem C in \cite{CT} depends on
an integral inequality due to Marcinkiewicz and Zygmund, and in Section 6 we
show that the published proof of this inequality is a Euclidean version of a
standard argument in hyperbolic geometry. We have already noted that the
exponent $1/2$ in Theorem A is best possible, but Kennedy and Twomey \cite{KT}
have shown how we can obtain more detailed information on the rate of growth
of ${\cal E}(r,\theta)$. See also Balogh and Bonk \cite{BB}.  We complete our
discussion of the known results in Sections 7 and 8 where we discuss their
result, a generalization of it to $L^p\!$-spaces, and some known estimates on
the basic scaling function $|f'(z)|/\lambda_\mathbb{D}(z)|$.  In the remaining
sections of the paper we raise and discuss the analogous issues for universal
cover maps instead of functions in the Dirichlet class. Finally, in the
Appendix, we give the proofs of some of the geometric results that we have
used in the earlier sections.

The authors are grateful for helpful comments by T. Carroll and J.B. Twomey on
an earlier draft of this paper.

\section{The hyperbolic metric on simply connected domains}

Each simply connected proper subdomain $D$ of $\mathbb{C}$ supports a
hyperbolic metric $\rho_D$ with density $\lambda_D$, where
\begin{equation}
\lambda_D\big(g(z)\big)|g'(z)| 
= \lambda_\mathbb{D}(z),
\end{equation}
and where $g$ is any conformal map of $\mathbb{D}$ onto $D$.  It is easy to
see that $\lambda_D$ and $\rho_D$ are independent of the choice of $g$, and
that $g$ is an isometry from $(\mathbb{D} ,\rho_\mathbb{D} )$ to $(D,\rho_D)$.
As $D$ is simply connected,
\begin{equation}
{1\over 2\dist [w,\partial D]} 
\ \leqslant\  \lambda_D (w) 
\ \leqslant\  {2\over \dist [w,\partial D]},
\end{equation}
where $\dist [w,\partial D]$ denotes the Euclidean distance from $w$ to the
boundary $\partial D$ of $D$ (these are standard estimates; see \cite{Ahlfors}
or \cite{BP}).

Next, $\lambda_D$ is continuous and positive on $D$. If $D$ is bounded, then
$\lambda_D(z) \to +\infty$ as $z$ approaches $\partial D$, and so $\lambda_D$
has a positive lower bound, say $\lambda_0$, on $D$. Thus if $f$ is bounded
and univalent in $\mathbb{D}$, and $D=f(\mathbb{D})$, then
\begin{equation}
0 
\ \leqslant\  {|f'(z)|\over \lambda_\mathbb{D}(z)} 
\ =\  {1\over \lambda_D\big(f(z)\big)} 
\ \leqslant\  {1\over \lambda_0},
\end{equation}
and 
\begin{equation}
{|f'(re^{i\theta})|\over \lambda_\mathbb{D}(re^{i\theta})} 
\to 0 \qquad \hbox{ as } \qquad r \to 1.
\end{equation}
These conclusions explain some of the earlier results in geometric terms
(albeit in the simpler case of bounded univalent maps). They show, for
example, that if $f$ is bounded and univalent on $\mathbb{D}$ then the map
$f:\mathbb{D}\to\mathbb{C}$ is Lipschitz with respect to the natural metrics
on $\mathbb{D}$ and $\mathbb{C}$.  Theorem 1.2 gives a stronger result than
this; however, this argument goes beyond Theorem 1.2. A domain $D$ is a {\it
Bloch domain} if $D$ does not contains arbitrary large Euclidean discs, and it
is clear from (2.2) that a simply connected domain $D$ is a Bloch domain if
and only if $\lambda_D$ has a positive lower bound on $D$. Thus {\sl if
$f:\mathbb{D}\to D$ is a univalent map of $\mathbb{D}$ onto a Bloch domain
$D$, then $f:\mathbb{D}\to\mathbb{C}$ is Lipschitz with respect to the natural
metrics on $\mathbb{D}$ and $\mathbb{C}$}. Of course, in this case $f$ need
not be in ${\cal D}$.

Finally, we note that (2.4) goes some way to explaining why ${\cal
E}(r,\theta)/\ell(r)^q$ may tend to zero as $r\to 1$.  As ${\cal E}(r,\theta)$
is the integral of $|f'|/\lambda_\mathbb{D}$ over a hyperbolic segment of
length $\ell(r)$, we certainly see that if (2.4) holds, then ${\cal
E}(r,\theta) = o\big(\ell(r)\big)$.  More generally, if the convergence in
(2.4) is sufficiently rapid, we might expect to get some result of the form
${\cal E}(r,\theta) = O\big(\ell(r)^q\big)$ as $r\to 1$.  As
$|f'(re^{i\theta})|$ may tend to $+\infty$ as $r\to 1$, it is perhaps harder
to see why the estimates of ${\cal E}(r,\theta)$ might hold if we take the
Euclidean point of view and integrate $|f'|$ with respect to Euclidean length
on $[0,e^{i\theta})$.

\section{A discussion of Theorem A}

Keogh's proof of Theorem A in \cite{Keogh} is short and elementary, and for
completeness we include it here. After this we give two examples to show that
the exponent $1/2$ is best possible, and we end with a geometric proof of a
similar result.
\medskip

\begin{proof}[of Theorem A]
It suffices to prove the result for ${\cal E}(r,0)$, and we write $f(z) =
\sum_n a_nz^n$. Then, for each positive integer $N$,
\begin{eqnarray*}
{\cal E}(r,0) &~=~& \int_0^r |f'(t)|\,dt\\
&~\leqslant~& \int_0^r \left(\sum_{n=1}^\infty n|a_n|t^{n-1}\right)\, dt\\
&~\leqslant~& \sum_{n=1}^\infty |a_n|r^n\\
&~\leqslant~& \sum_{n=1}^{N-1}|a_n| + \sum_{n=N}^\infty |a_n|r^n\cr
&~\leqslant~& \sum_{n=1}^{N-1}|a_n| + 
\left(\sum_{n=N}^\infty n|a_n|^2\right)^{1/2}
\left(\sum_{n=N}^\infty{r^{2n}\over n}\right)^{1/2}\\
&~\leqslant~& \sum_{n=1}^{N-1}|a_n| + 
\left(\sum_{n=N}^\infty n|a_n|^2\right)^{1/2}
\sqrt{\log {1\over 1-r^2}}.\\
\end{eqnarray*}
As $\sum_n n|a_n|^2$ converges, (1.2) follows from this.
\end{proof}

We continue our discussion of Theorem A by providing the following conformally
invariant version of it. It is essential that we use the conformally invariant
hyperbolic length here rather than the classical bound $\log (1-r)^{-1}$.

\begin{theorem}
Suppose that $\Omega$ is a domain that is conformally equivalent to
$\mathbb{D}$, and that $f$ is analytic in $\Omega$ with $f(\Omega)$ of finite
area $A(f)$. Then, for any hyperbolic geodesic segment $\gamma$ in $\Omega$,
\[
{\cal E}\big(f(\gamma)\big) 
\ \leqslant\  
\sqrt{{A(f)\over \pi}{\cal H}(\gamma)},
\]
where ${\cal E}\big(f(\gamma)\big)$ is the Euclidean length of $f(\gamma)$,
and ${\cal H}(\gamma)$ is the hyperbolic length of $\gamma$.  In particular,
if $\Omega$ has finite area $A$, then $\pi {\cal E}(\gamma)^2 \leqslant A{\cal
H}(\gamma)$.
\end{theorem}

\begin{proof*}
This is easy. We can find a conformal map $g$ of $\Omega$ onto $\mathbb{D}$ in
such a way that $g$ maps $\gamma$ onto some real segment $[0,r]$ of
$\mathbb{D}$, and then apply Keogh's result to $F$ defined by $F=f\circ
g^{-1}$.  If we take $N=1$ (and $f=F$) in the proof given above, we obtain
\[
{\cal E}_F(r,0) 
\ \leqslant\  
\sqrt{{A(F)\over \pi}\log{1\over 1-r^2}},
\]
and this gives the stated inequality because $A(f)=A(F)$, ${\cal E}_F(r,0) =
{\cal E}\big(f(\gamma)\big)$ and
\[ 
\singlebox
\log{1\over 1-r^2}
\ \leqslant\  \log{1+r\over 1-r} 
\ =\  \ell(r) 
\ =\ {\cal H}(\gamma).
\esinglebox 
\]
\end{proof*}

The estimate $O(\sqrt{{\cal H}(\gamma)})$ is universally valid over all
geodesic segments $\gamma$.  However, if we take $\gamma$ to be the initial
segment of any geodesic ray starting from a fixed point of $\Omega$ then,
exactly as in Keogh's result, we can show that ${\cal E}\big(f(\gamma)\big) =
o(\sqrt{{\cal H}(\gamma)})$. This will be considered further in Section 4.

Theorem A applies to all functions in the Dirichlet space ${\cal D}$. However
in \cite{Keogh}, Keogh showed that the exponent $1/2$ in Theorem A is best
possible when the result is restricted to the smaller class of functions that
are bounded and univalent in $\mathbb{D}$.  Further examples were given by
Jenkins \cite{Jenkins}, Kennedy \cite{Kennedy}, and Carroll and Twomey
\cite{CT}, to show that not only is the exponent $1/2$ in Theorem A best
possible, but so is the function on the right hand side of (1.2), and we shall
look at these results in Section 7. Keogh's example in \cite{Keogh} is based
on the function $g(z)^{2\alpha -1} + i \sin\big(g(z)^\alpha\big)$, where $g$
is a conformal map of $\mathbb{D}$ onto a half-infinite strip, and $\alpha$ is
a suitable constant, and his argument is technical, and not self-contained. In
our examples (Examples 3.2 and 3.3) we study the analytic properties of a
univalent map $g:\mathbb{D}\to D$ indirectly through the hyperbolic geometry
of $g(\mathbb{D})$.  Example 3.2 is extremely simple, but the function in this
example is unbounded. Example 3.3 is of a bounded univalent function in ${\cal
D}$.

\begin{example}
Suppose that $\beta$ is a constant such that for all $f$ in ${\cal D}$, and
all real $\theta$, we have ${\cal E}(r,\theta) = O(\ell(r)^\beta )$ as $r\to
1$; we shall show that $\beta \geqslant \frac{1}{2}$. Take any positive
$\varepsilon$, and let $D = \{x+iy:x >1,\ |y|<1/x^{1+\varepsilon}\}$.  It is
clear that if $t \geqslant 2$, then the open rectangle
\[
\{x+iy: 1 < x < t+1, \quad |y| < 1/(t+1)^{1+\varepsilon}\}
\]
lies in $D$, and from this we see that
\[
\lambda_D(t) 
\ \leqslant\  {2\over {\rm dist}[t,\partial D]}
\ \leqslant\   2(t+1)^{1+\varepsilon}.
\]
Now let $f$ be the conformal map of $\mathbb{D}$ onto $D$ that maps $[0,1)$
onto $[2,+\infty)$. Clearly, $f\in {\cal D}$.  Next, as $f:\mathbb{D}\to D$ is
a hyperbolic isometry,
\[
\ell(r) \ =\  \int_2^{f(r)} \lambda_D(t)\, dt
\ \leqslant\  2\int_2^{f(r)} (t+1)^{1+\varepsilon}\,dt
\ <\  [f(r)+1]^{2+\varepsilon}.
\] 
Finally, as $f$ is increasing on $(0,1)$ it maps $[0,r]$ onto $[2,f(r)]$ so
that ${\cal E}(r) = f(r)-2$, and hence $\ell(r) \leqslant ({\cal
E}(r)+3)^{2+\varepsilon}$.  As ${\cal E}(r,\theta) = O(\ell(r)^\beta )$, it
follows that $1 \leqslant \beta (2+\varepsilon)$. As $\varepsilon$ is
arbitrary we have $\beta \geqslant \frac{1}{2}$.  The reader will notice that,
in this example, the conclusion follows once we know that ${\cal E}(r,\theta)
= O(\ell(r)^\beta )$ as $r\to 1$ for a single value of $\theta$.
\end{example}

\begin{example}
This example is similar to that given in \cite[p.\,491]{CT} to show that there
is a function $f$ that is univalent, bounded, close-to-convex and in ${\cal
D}$ such that the image curve $f\big([0,1)\big)$ may be of infinite
length. Here we let $f$ be a conformal map of $\mathbb{D}$ onto a particular
bounded simply connected domain $D$ which we shall now construct.

\begin{center}
\epsfig{file=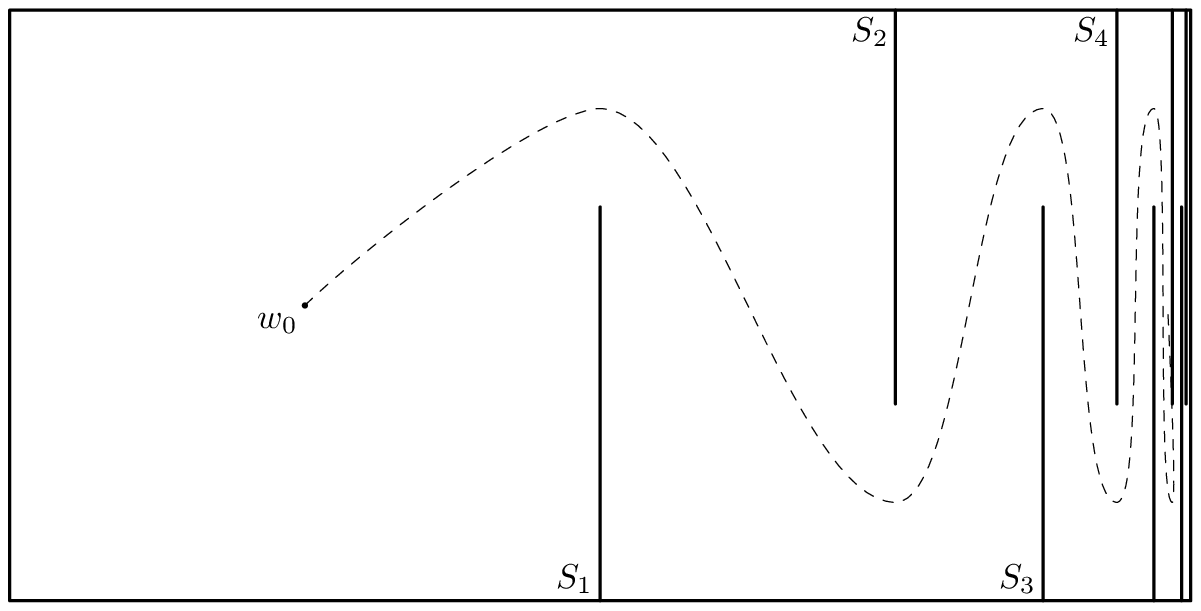}
\end{center}

\centerline{{\bf Figure 1.}\ (The geodesic $f([0, e^{i\theta}))$ is
shown as a dashed line.)} 
\bigskip

Suppose that $\varepsilon >0$ and for $n=1,2,\ldots$, let
\[
a_n = {1\over n^{1+\varepsilon}}, 
\qquad s_n = a_1+\cdots +a_n ,
\qquad s=\sum_{n=1}^\infty a_n.
\]
Now consider the rectangle
\[
R = \{x+iy : 0<x<s, \ -3/2<y<3/2\},
\]
and form a simply connected domain $D$ by deleting from $R$ the vertical slits
\[
S_n = \cases{
\{s_n +iy : -3/2 \leqslant y \leqslant 1/2\} 
&if $n= 1,3,5,\ldots$,\cr
\{s_n +iy : -1/2 \leqslant y \leqslant 3/2\} 
&if $n=2,4,6,\ldots$\ ;}
\]
this is illustrated in Figure 1.  For $n=1,2,\ldots$, let $L_n$ be the segment
of the $y$-axis that joins the slits $S_n$ and $S_{n+1}$; let $w_n$ be
the mid-point of $L_n$ for $n\geqslant 1$; and $w_0 = \frac{1}{2}$.

Now let $f$ be the unique univalent map of $\mathbb{D}$ onto $D$ with
$f(0)=w_1$ and $f'(0)>0$. Then $f \in {\cal D}$, and there is some real
$\theta$ such that the image curve $f\big([0,e^{i\theta})\big)$ accumulates
at, and only at, the right hand edge of $R$ (see \cite[p.\,491]{CT} and
\cite{Pommerenke}). Using this value of $\theta$ we shall construct a sequence
of values $r_n$ with $0=r_0 < r_1<\cdots <r_n<\cdots$,\  $r_n \to 1$,\
$f(r_n)\in L_n$ and
\begin{equation}
{\cal E}(r_n,\theta) \geqslant n, \quad
\ell(r_n) =O(n^{2+\varepsilon}).
\end{equation}

Let us assume for the moment that we have done this. Now suppose that $\beta$
is such that
\begin{equation}
{\cal E}(r,\theta) = O(\ell(r)^\beta)
\qquad \hbox{ as } \qquad r \to 1 
\end{equation} 
for every univalent bounded function in ${\cal D}$ and every
choice of $\theta$; then this holds for the function $f$ constructed above,
and with $r=r_n$.  Now (3.1) and (3.2) imply that $1 \leqslant
(2+\varepsilon)\beta$, and as this is true for every positive $\varepsilon$ we
again conclude that $\beta \geqslant \frac{1}{2}$ as required. It remains only
to construct the sequence $r_n$ satisfying (3.1) for the function $f$ and the
value $\theta$ described above.

We put $r_0=0$ and, for $n \geqslant 1$, we let $r_n$ be the smallest value of
$r$ in $(0,1)$ at which the point $f(re^{i\theta})$ lies on $L_n$.  It is
evident that $0=r_0 < r_1 < r_2 < \cdots < r_n$, that $r_n \to 1$, and that
${\cal E}(r_n,\theta) \geqslant n$; this last inequality is the first
inequality in (3.1).

It remains to establish the estimate for $\ell(r_n)$ in (3.1).  As $f$ is a
conformal map of $\mathbb{D}$ onto $D$, it is a hyperbolic isometry, so that
$\ell(r_n)$ is the hyperbolic length of the curve $f(te^{i\theta})$,
$0\leqslant t \leqslant r_n$.  As this curve is a geodesic with end-points
$w_0$ and $f(r_ne^{i\theta})$, its length is at most the hyperbolic length of
any curve, say $\Gamma$, in $D$ that passes through the points $w_0,
w_1,\ldots ,w_n,f(r_ne^{i\theta})$ in this order.  It suffices, then, to show
that there is a choice of $\Gamma$ whose length is $O(n^{2+\varepsilon})$.  We
take $\Gamma$ to be a polygonal curve constructed as follows. For $k=0,
1,\ldots, n-1$, we join $w_k$ to $w_{k+1}$ by a curve $\Gamma_k$ in $D$
comprising a vertical segment of Euclidean length one , followed by
a horizontal segment (from left to right), and then by another vertical
segment of Euclidean length one. We now let $\Gamma$ be the curve
traced out by $\Gamma_0, \Gamma_1,\ldots ,\Gamma_{n-1}$, followed by the
horizontal segment $\Gamma_n^*$ of $L_n$ that joins $w_n$ to
$f(r_ne^{i\theta})$.  Thus if we denote the hyperbolic length of a curve
$\sigma$ in $D$ by $h_\ell(\sigma$), we have
\[
\ell(r_n) 
\ \leqslant\  h_\ell(\Gamma_0) + h_\ell(\Gamma_1) + \cdots +
h_\ell(\Gamma_{n-1})  + h_\ell(\Gamma_n^*).
\]

We now estimate the terms on the right hand side of this inequality. As each
point of $\Gamma_k$ is at least a distance $d_k$ away from $\partial D$, where
\[
d_k \ =\  \frac{a_{k+2}}{2} \ =\  {1\over 2(k+2)^{1+\varepsilon}},
\]
we see from (2.2) that $h_{\ell}(\Gamma_k) \leqslant
12(k+2)^{1+\varepsilon}$. Thus
\begin{eqnarray*}
\ell(r_n) 
&~\leqslant~& 12\sum_{m=3}^{n+1} m^{1+\varepsilon} + h_\ell(\Gamma_n^*)\\
&~\leqslant~& 12\int_3^{n+2}x^{1+\varepsilon}\,dx + h_\ell(\Gamma_n^*)\\
&~<~& 6(n+2)^{2+\varepsilon} + h_\ell(\Gamma_n^*).\\
\end{eqnarray*}
Our discussion of this example will be complete if we can show that for some
positive number $M$, and all $n = 1,2,\ldots$,
\begin{equation}
h_\ell(\Gamma_n^*) \leqslant M.
\end{equation}
In fact, $h_\ell(\Gamma_n^*) \to 0$ as $n\to\infty$, because if a domain $D$
contains a long thin rectangle whose long sides lie in $\partial D$, then any
hyperbolic geodesic in $D$ that travels the length of the rectangle must pass
very close to the centre of the rectangle. (See \cite{Beardon1} and
\cite{Beardon2}.) We do not need as much as this, and the following result
(whose proof is given in the Appendix) yields (3.3) immediately, and thereby
completes our discussion of Example 3.3.  Note that the hyperbolic metric is
unchanged by (Euclidean) translations and scalings of a domain.  So we may
assume that the rectangle $R$ is centred at the origin and of a fixed width.
\end{example}

\begin{theorem}
There is a positive number $M$ with the following property.  Let $R$ be the
square $(-1, 1)\times (-1, 1)$ and let $D$ be a simply-connected, hyperbolic
domain such that $R \subset D \subset \mathbb{C}$, and the two sides of $R$
parallel to the $y$-axis lie in $\partial D$.  If $\gamma$ is any hyperbolic
geodesic in $D$ that does not have an endpoint on the sides of $R$ parallel to
the $y$-axis, then it can only meet the crosscut $(-1,1)\times \{0\}$ of
$D$ at a hyperbolic distance of at most $M$ from the origin.
\end{theorem}

We complete this discussion of Theorem A by returning to the conformally
invariant version of it, namely Theorem 3.1.  This says that if $\gamma$ is a
hyperbolic geodesic segment in a simply connected domain $\Omega$ of finite
area $A$, then $\pi {\cal E}(\gamma)^2 \leqslant A{\cal H}(\gamma)$.  The
following result shows that a similar inequality also holds for Euclidean
geodesic segments; note that in this result {\it the domain need not be
simply-connected}.

\begin{theorem}.
Let $D$ be a domain in $\mathbb{C}$ with finite area, and let $\gamma (t)$,
$t\geqslant 0$, be a Euclidean ray in $D$. Then the Euclidean and hyperbolic
lengths ${\cal E}(t)$ and ${\cal H}(t)$, respectively, of the segment
$[\gamma(0),\gamma(t)]$ satisfy ${\cal E}(t) = o({\cal H}(t)^{1/2})$ as $t\to
+\infty$.
\end{theorem}

\begin{proof}
We may assume that $\gamma$ is the positive real axis, and that $\gamma(t)=t$
for $t \geqslant 0$. In \cite{Solynin} Solynin shows that if $D$ is a domain
in $\mathbb{C}$ that meets the real axis, and if $\tilde{D}$ is the domain
obtained from $D$ by symmetrising it about $\mathbb{R}$, then the two
hyperbolic densities satisfy $\lambda_D(x) \geqslant \lambda_{\tilde{D}}(x)$
for $x\in D \cap \mathbb{R}$. We let $m(t)$ be the Euclidean length of the
intersection of $D$ with the vertical line $x=t$; then $\tilde{D}$ is given by
$|y| < \frac{1}{2} m(x)$, and if $t \geqslant 0$ then
\[ 
\lambda_\mathbb{D}(t) 
\ \geqslant\  \lambda_{\tilde{D}}(t)
\ \geqslant\  {1\over 2{\rm dist}[t,\partial\tilde{D}]}
\ \geqslant\  {1 \over 2(\frac{1}{2} m(t))} = {1\over m(t)}.
\]
This implies that the hyperbolic length ${\cal H}(r)$ of $[0,r]$ satisfies
\[ 
{\cal H}(r) \ \geqslant\  \int_0^r {1\over{m(t)}}\,dt.
\]

Next, the Euclidean length ${\cal E}(r)$ of $[0,r]$ is $r$ so, using the
Cauchy-Schwarz inequality, we have
\[
{\cal E}(r)^2
\ =\ \left(\int_0^r 1\,dt\right)^2 
\ \leqslant\  \left(\int_0^r m(t)\, dt\right) 
\left(\int_0^r {1\over{m(t)}}\,dt\right)
\ \leqslant\  \left(\int_0^r m(t)\,dt\right) {\cal H}(r).
\]
More generally, for any positive $c$, and any $r$ with $r>c$, we have
\[
{\cal E}(r) 
\ \leqslant\  c + \int_c^r 1\,dt 
\ \leqslant\  c + \left(\int_c^r m(t)\,dt\right)^{1/2} {\cal H}(r)^{1/2}.
\]
As the integral $\int_{-\infty}^\infty m(t)\,dt$ is the Euclidean area of the
domain $D$, and this is finite, we see that
\[
\int_c^\infty m(t)\,dt \to 0 \qquad \hbox{ as } \qquad c \to +\infty ;
\]
thus ${\cal E}(r) = o({\cal H}(r)^{1/2})$ as required.
\end{proof}

Theorems 3.1 and 3.5 raise the question of characterising those curves
$\gamma$ in a given domain whose Euclidean and hyperbolic lengths satisfy some
inequality of the form ${\cal E}(\gamma) \leqslant M{\cal H}(\gamma)^{1/2}$,
or some variant of this. We can show that this conclusion also holds for
curves that are in some sense sufficiently close to Euclidean rays.  Let
$\gamma:[0,+\infty)\to \mathbb{C}$ be a smooth curve in $\mathbb{C}$
parametrised by Euclidean arc length, and suppose that $\gamma$ is
sufficiently close to an Euclidean ray so that the map $\gamma$ extends to a
quasi-conformal map $g$ that maps a neighbourhood $W$ of $[0,\infty)$ of
finite area onto a neighbourhood $D$ of $\gamma$. Now Theorem 3.5 applies to
the ray $[0,\infty)$ in the domain $W$. However, $g$ does not change Euclidean
length along $[0,\infty)$ and, since $g$ is quasi-conformal, it only alters
hyperbolic length by at most a fixed factor; thus the same conclusion holds
for $\gamma$ in $D$.

We end this section by showing that Theorem 3.5 contains Theorem A in the
special case of conformal maps. Suppose that $f$ is a conformal map of
$\mathbb{D}$ onto a domain $D$ of finite area. We write $f(z) = \sum_n a_n
z^n$, and define $g:\mathbb{D}\to\mathbb{C}$ by $g(z)= \sum_n |a_n| z^n$. Note
that $g$ maps $[0,1)$ into the positive real axis,
\[
A(g) \ =\  \sum_{n=0}^\infty n|a_n|^2 \ =\  A(f) \ <\  +\infty,
\]
and the Euclidean lengths of $f([0,r])$ and $g([0,r])$ satisfy
\[
{\cal E}_f(r,0) 
\ =\  \int_0^r |f'(t)|\,dt 
\ \leqslant\ \int_0^r g'(t)\,dt 
\ =\  {\cal E}_g(r,0).
\]
If ${\cal E}_g(r,0)$ is finite there is nothing to prove.  If not, then
$g(\mathbb{D})$ contains the interval $[|a_0|, \infty)$ and we can apply
Theorem 3.5 to $g$ and obtain
\[
{\cal E}_g(r,0) = o\big(L(r)^{1/2}\big),
\]
where $L(r)$ is the hyperbolic length of $g[0,r]$ in $g(\mathbb{D})$.  The
Schwarz-Pick lemma applied to $fg^{-1} : g(\mathbb{D}) \to f(\mathbb{D})$
implies that $L(r) \leqslant \ell(r)$ so we obtain the conclusion of Theorem A
for $f$, namely that ${\cal E}_f(r,0) = o(\ell(r)^{1/2})$ as $r\to 1$.

\section{A proof of Theorem B}

We shall now show how the apparently stronger Theorem B can be deduced
directly from Theorem A by a geometric argument that is more transparent than
the analytic proof of Theorem B given in \cite{CT}.  The idea is to apply
Theorem A to the restriction of $f$ to the Stolz region $S$ contained in
$\mathbb{D}$.  On the radial line $[0, e^{i\theta})$, the hyperbolic metric
$\rho_S$ for $S$ and the hyperbolic metric for $\mathbb{D}$ are comparable, so
Theorem A for $f|_S$ implies Theorem B.  The details of this proof rely on the
hyperbolic geometry of Stolz regions, so we write it out fully.

\begin{theorem}.
Let $f : \mathbb{D} \to \mathbb{C}$ be any analytic function.  For any two
points $z_0, z_1$ in $\mathbb{D}$, let $[z_0, z_1]$ be the hyperbolic geodesic
segment from $z_0$ to $z_1$ and let $ \Omega = \{ z\in \mathbb{D} :
\rho_\mathbb{D}(z, [z_0, z_1]) < d \} $ for a fixed constant $d > 0$.  Then
\[
{\cal E}(f([z_0, z_1])) 
\ \leqslant\  
\left({{A(f(\Omega))}\over{\pi d_E}}\right)^{1/2}
\rho_\mathbb{D}(z_0, z_1)^{1/2}
\]
where $d_E = \tanh \frac{1}{2} d$.
\end{theorem}

\begin{proof*}
First observe that $\Omega$ is a simply connected subdomain of $\mathbb{D}$.
Also, $\Omega$ is hyperbolically symmetric about the geodesic through $z_0$
and $z_1$, so $[z_0, z_1]$ is also a geodesic segment for the hyperbolic
metric $\rho_\Omega$ on $\Omega$.  We wish to compare the two hyperbolic
metrics $\rho_\mathbb{D}(z_0, z_1)$ and $\rho_\Omega(z_0, z_1)$.

Consider the ratio $\lambda_\Omega(z)/\lambda_\mathbb{D}(z)$ for $z\in [z_0,
z_1]$.  We can apply a M\"obius isometry to move $z$ to the origin.  Then
$\Omega$ contains the disc $\{z : |z| < d_E \}$ with hyperbolic radius $d$ and
Euclidean radius $d_E = \tanh \frac{1}{2} d$.  Consequently,
$\lambda_\Omega(0) \leqslant \lambda_{\mathbb{D}(0, d_E)}(0) = 2/d_E$ and so
$\lambda_\Omega(z)/\lambda_\mathbb{D}(z) \leqslant 1/d_E$ (cf. (2.2)).
Integrating this along the curve $[z_0, z_1]$ relative to the hyperbolic arc
length $ds_\mathbb{D}$ in $\mathbb{D}$ shows that
\[
\rho_\Omega(z_0, z_1) 
\ =\  \int_{[z_0, z_1]}
\lambda_\Omega(z)/\lambda_\mathbb{D}(z) \; ds_\mathbb{D}(z) 
\ \leqslant\  {{\rho_\mathbb{D}(z_0, z_1)}\over{d_E}}\ .
\]

Theorem 3.1 is an invariant form of Theorem A, so we can apply it to
the restriction \hbox{$f|_\Omega : \Omega \to \mathbb{C}$}.  This gives
\[
{\cal E}(f([z_0, z_1])) 
\ \leqslant\  \left({{A(f(\Omega))}\over\pi}\right)^{1/2}
\rho_\Omega(z_0, z_1)^{1/2}\ .
\]
Therefore we see that
\[
\singlebox
{\cal E}(f([z_0, z_1])) 
\ \leqslant\  
\left({{A(f(\Omega))}\over{\pi d_E}}\right)^{1/2}
\rho_\mathbb{D}(z_0, z_1)^{1/2}\ .
\esinglebox
\]
\end{proof*}

This Theorem is enough to show that ${\cal E}(r, 0) = O(\rho^{1/2})$
but we want the slightly stronger result ${\cal E}(r, 0) =
o(\rho^{1/2})$.  We could deduce this by applying Theorem A rather
than Theorem 3.1.  However, we give an alternative argument below.

Fix $z_0$ at the origin and let $z_1$ be a point $r\in [0,1)$.  We
will write $\rho$ for $\rho_\mathbb{D}(0, r) = \log (1+r)/(1-r)$.  Let
$\gamma$ be the half-line $[0, 1)$ and let $\Sigma$ be the region 
$\{z\in \mathbb{D} : \rho(z, \gamma) < d \}$.
(This is actually a Stolz region with vertex at $1$.)  
More generally, let $\Sigma(r) = 
\{ z\in \mathbb{D} : \rho_\mathbb{D}(z, [r,1)) < d \}$.  If the area
$A(f(\Sigma))$ is finite, then 
\[
A(f(\Sigma(r))) \searrow 0 \qquad \hbox{ as } \ r\nearrow 1\ .
\]
For any fixed $r_0\in [0,1)$ we can apply Theorem 4.1 to show that
\begin{eqnarray*}
{\cal E}(f([0, r])) 
&~\leqslant~&  {\cal E}(f([0, r_0])) + {\cal E}(f([r_0, r])) \\
&~\leqslant~&  {\cal E}(f([0, r_0])) + 
\left({{A(f(\Sigma(r_o)))}\over{\pi D}}\right)^{1/2} 
\rho_\mathbb{D}(r_0, r)^{1/2} \\
&~\leqslant~&  {\cal E}(f([0, r_0])) + 
\left({{A(f(\Sigma(r_o)))}\over{\pi D}}\right)^{1/2} 
\rho^{1/2} \ .\\
\end{eqnarray*}
This shows that ${\cal E}(f([0, r])) = o(\rho^{1/2})$ and so proves
Theorem B when the Stolz region is $\Sigma$.  It remains to show that
any Stolz region $S$ with vertex at $1$ contains $\Sigma$ provided
that $d$ is chosen small enough.

Let $S$ be a Stolz region with vertex at $1$ and $A(f(S))$ finite.
Since $f$ is bounded on any finite disc $\{z : |z|< R \}$ of radius
$R<1$, we may assume that $S$ contains the entire half-line
$[0,1)$. The fact that $S$ is a Stolz region means that there
is a constant $c>0$ with 
\[
{\rm dist}[r, \partial S] \geqslant c (1-r) 
\qquad \hbox{ for each }\qquad r\in [0,1) .
\]
For each $z$ in the disc $\{z : |z-r| < \frac{1}{2}
c(1-r) \}$ we have $1-|z| \leqslant (1-r) + |z-r| < (1+\frac{1}{2} c) (1-r)$ so
\[
\lambda_\mathbb{D}(z) \ =\  {2\over{1-|z|^2}} 
\ \geqslant\  {1\over{1-|z|}} 
\ >\ {1\over{(1+\frac{1}{2} c)(1-r)}}.
\]
This means that the hyperbolic distance $\rho_\mathbb{D}(r, \partial S)$
from $r$ to the boundary of $S$ satisfies
\[
\rho_\mathbb{D}(r, \partial S) 
\ \geqslant\  {{\frac{1}{2} c(1-r)}\over{(1+\frac{1}{2} c)(1-r)}} 
\ =\  {c\over{c+2}}.
\]
Hence $S$ contains the region $\Sigma = \{z : \rho_\mathbb{D}(z, [0,1))
< d \}$ provided that $d \leqslant c/(c+2)$.  Since $A(f(S))$ is finite, so
is $A(f(\Sigma))$.  We have already proved that, in this case, ${\cal
E}(r,1) = o(\rho^{1/2})$ so the proof of Theorem B is complete.

Note that we also have a version of these inequalities for functions
that are not in the Dirichlet class.  Write 
$A(\rho) = A(f(\mathbb{D}(0, \tanh \frac{1}{2} \rho)))$ for the
area of the image under $f$ of the disc $\mathbb{D}(0, \tanh \frac{1}{2}
\rho)$ with hyperbolic radius $\rho$.  For the geodesic $[0, r]$, the
region $\Omega$ is a subset of the disc $\mathbb{D}(0, \tanh
\frac{1}{2}(\rho+d))$.  Hence Theorem 4.1 gives
\[
{\cal E}(f([0, r])) 
\ \leqslant\  \left({{A(\rho + d)}\over{\pi D}}\right)^{1/2} \rho^{1/2} .
\]

\section{Theorem C}

In this section we give a simple geometric argument to show how
Theorem C follows directly from Theorem B and the distortion theorem
for univalent functions. The proof in \cite{CT} of the main inequality in
Theorem C depends crucially on Theorem D, which is stated and proved
geometrically in the next section.  We will first prove the following
slightly modified form of Theorem C. 

\begin{theorem}
There is a constant $C$ such that, if $f$ is analytic and univalent in
$\mathbb{D}$ with $w\notin f(\mathbb{D})$, then 
\[
{\cal E}(r,\theta) 
\ \leqslant\ C
\sup\{|f(te^{i\theta})-w| : 0\leqslant t\leqslant r\}
\left(\log{{1+r}\over{1-r}}\right)^{1/2}
\]
for each $\theta$ and each $r\in [0,1)$.
\end{theorem}

This certainly implies that ${\cal E}(r, \theta) = o(\rho^{1/2})$ when
$f$ is bounded on the half-line $[0,e^{i\theta})$.

We begin with a simple geometric result.  This is a version of the
distortion theorem for univalent functions.

\begin{lemma}
Let $D$ be a simply connected subdomain of $\mathbb{C}$,
and suppose that $0\notin D$. Then for any two points $w_1$ 
and $w_2$ in $D$ with $|w_1| \leqslant |w_2|$, we have
\[
\rho_D(w_1,w_2) \geqslant 
{\textstyle{1\over 2}}\log {|w_2|\over |w_1|}.
\]
\end{lemma}

\begin{proof}
As $D$ is simply connected and $0\notin D$, if $z\in D$ then
\[
\lambda_D(z) \ \geqslant\  {1\over 2{\rm dist}[z,\partial D]} 
\ \geqslant\ {1\over 2|z|}.
\]
Now let $\sigma$ be the hyperbolic geodesic in $D$ joining $w_1$ 
to $w_1$. Then
\begin{eqnarray*}
2\rho_D(w_1,w_2) 
&~=~& 2\int_\sigma \lambda_D(z)\, |dz| 
\ \geqslant\ \int_\sigma {|dz|\over |z|}\\
&~\geqslant~& \left|\int_\sigma {dz\over z}\right|
\ \geqslant\  \left|{\rm Re}\left[\int_\sigma {dz\over z}\right]\right|
\ =\ \log {|w_2|\over |w_1|}\\
\end{eqnarray*}
and the proof is complete.
\end{proof}

\begin{proof}[of Theorem 5.1]
Let $\Omega = \{z\in \mathbb{D} : \rho_\mathbb{D}(z, [0,re^{i\theta}]) < d \}$
be the neighbourhood of the geodesic segment $[0, re^{i\theta}]$ as in the
previous section.  If we show that $f(\Omega)$ lies within a disc of radius
$e^{2d} \sup\{|f(te^{i\theta})-w| : 0\leqslant t\leqslant r\}$, then
\[
A(f(\Omega)) \ \leqslant\  \pi e^{4d} 
\left(\sup\{|f(te^{i\theta})-w| : 0\leqslant t\leqslant r\}\right)^2 .
\]
So Theorem 4.1 will show that
\begin{eqnarray*}
{\cal E}(r,\theta) &~\leqslant~& 
\left( {{A(f(\Omega))}\over{\pi D}} \right)^{1/2} \rho^{1/2} \\
&~\leqslant~& 
{{e^{2d}}\over{D^{1/2}}} \sup\{|f(te^{i\theta})-w| : 0\leqslant t\leqslant r\} 
\;\rho^{1/2} \\
\end{eqnarray*}
and complete the proof.

For each point $z\in \Omega$, there is a $\zeta\in [0, re^{i\theta}]$ with
$\rho_\mathbb{D}(z, \zeta) < d$.  The Schwarz-Pick lemma shows that $f :
\mathbb{D} \to f(\mathbb{D})$ decreases hyperbolic lengths.  So
$\rho_{f(\mathbb{D})}(f(z), f(\zeta)) < d$.  Lemma 5.2 now shows that
$|f(z)-w| < |f(\zeta)-w| e^{2d}$.
Consequently,
\[
|f(z)-w| \ <\  e^{2d} \sup \{|f(te^{i\theta})-w| : t\in [0, r]\} .
\]
This means that $f(\Omega)$ lies within a disc centred on $w$ and with radius
$e^{2d} \sup \{|f(te^{i\theta}) - w| : 0\leqslant t\leqslant r\}$, as
required.
\end{proof}

Theorem 5.1 is certainly strong enough to show that ${\cal E}(r,
\theta) = o(\rho^{1/2})$ when $f$ is bounded on the half-line
$[0,e^{i\theta})$.   However, for completeness, we will show that Theorem 5.1
implies Theorem C.  Suppose that $f : \mathbb{D} \to \mathbb{C}$ is
analytic and univalent with $f(0) = 0$. 
Choose $w$ to be a point of $\mathbb{C}\setminus f(\mathbb{D})$ with minimal
modulus, and let $\zeta = f(\frac{1}{2} e^{i\theta})/w$.  We will
show that $|\zeta| \geqslant \frac{1}{2}$.  We may suppose that $\zeta
\in \mathbb{D}$ as, otherwise, this is obviously true.  Then, as $z\mapsto wz$
is a hyperbolic contraction of $\mathbb{D}$ into $f(\mathbb{D})$, we have 
\[
\rho_\mathbb{D}(0, \zeta) 
\ \geqslant\ \rho_{f(\mathbb{D})}(0, w\zeta) 
\ =\ \rho_{f(\mathbb{D})}(f(0), f(\frac{1}{2} e^{i\theta}))
\ =\ \rho_\mathbb{D}(0, \frac{1}{2} e^{i\theta}) 
\]
so that $|\zeta| \geqslant {\textstyle 1\over 2}$ in this case too.

Since $|f(\frac{1}{2} e^{i\theta})| \geqslant \frac{1}{2}|w|$, we now have 
\begin{eqnarray*}
\sup \{ |f(te^{i\theta})-w| : 0\leqslant t\leqslant r \} 
&~\leqslant~&
\sup \{ |f(te^{i\theta})| : 0\leqslant t\leqslant r \} \ +\ |w|  \\
&~\leqslant~&
\sup \{ |f(te^{i\theta})| : 0\leqslant t\leqslant r \} \ +\ 2|f(\frac{1}{2}
e^{i\theta})| .\\
\end{eqnarray*}
Provided that $r\geqslant \frac{1}{2}$, the right side is certainly bounded by 
$3 \sup \{ |f(te^{i\theta})| : 0\leqslant t\leqslant r \}$ and so Theorem C
follows from 5.1.

We conclude this discussion of Theorem C with the following 
variant of it. First, we introduce the spaces $\Lambda^p$ 
of measurable functions on $[0,1)$ whose $p\!$-th power 
is integrable with respect to the hyperbolic length
$\lambda_\mathbb{D}(t)\,dt$.

\begin{theorem}
In the notation above, if 
$f'(re^{i\theta})/\lambda_\mathbb{D}(r) \in \Lambda^p$ then
\begin{equation}
{\cal E}(r,\theta) = o\big(\ell(r)^{1/q}\big)
\end{equation}
as $r\to 1$, where $1/p+1/q=1$.
In particular, if $f'/\lambda_\mathbb{D}$ is square-integrable
with respect to hyperbolic length along the ray ending
at $e^{i\theta}$, then 
${\cal E}(r,\theta) = o(\sqrt{\ell (r)})$ as $r\to 1$.
\end{theorem}

\begin{proof}
We shall only consider the radius $(0,1)$ but a similar 
argument will hold for any radius. Suppose that $f$ 
is analytic in $\mathbb{D}$, and take any $p$ and $q$ in 
$[1,+\infty]$ with $1/p +1/q =1$. H\"older's inequality 
gives
\begin{eqnarray*}
{\cal E}(r,0)
&~=~& \int_0^r \left({|f'(t)|\over \lambda_\mathbb{D} (t)^{1/q}}\right)
\lambda_\mathbb{D} (t)^{1/q}\, dt\\
&~\leqslant~& 
\left(\int_0^r \left({|f'(t)|\over \lambda_\mathbb{D} (t)}\right)^p 
\lambda_\mathbb{D} (t)\,dt\right)^{1/p}\ell(r)^{1/q}\\
\end{eqnarray*}
so that if the function 
$r \mapsto f'(re^{i\theta})/\lambda_\mathbb{D}(r)$ is in 
$\Lambda^p$ then (5.1) holds with a $O$-estimate instead of 
a $o$-estimate. The conversion from the $O$-term to the 
$o$-term is carried out as in the proof of Theorem A in 
Section 3.
\end{proof}

\section{An integral inequality}

The proof of Theorem C given in \cite{CT} depends on the following 
inequality due to Marcinkiewicz and Zygmund 
(Theorem 2, \cite[p.\,479]{MZ}).

\begin{theorem D*}
For each Stolz region $S$ in $\mathbb{D}$ at the point 
$e^{i\theta}$ there is a number $\beta(S)$ such that if 
$f$ is analytic in $\mathbb{D}$ then
\begin{equation}
\int_0^1 (1-r)|f'(re^{i\theta})|^2\,dr
\ \leqslant\  \beta(S) {\rm area}\big(f(S)\big).
\end{equation}
\end{theorem D*}

There is no discussion in \cite{MZ} of 
why one might expect the factor $1-r$ to appear in the 
integrand in (6.1), but a close examination of the proof there 
shows that it is an Euclidean version of what is a completely 
standard geometric construction in hyperbolic geometry. For 
this reason one should not be surprised at the appearance 
of such a factor. Consistent with the general point of view 
expressed in this article, we suggest that Theorem D (and 
its proof) should be rewritten in the following form.

\begin{theorem D*}
For each Stolz region $S$ in $\mathbb{D}$ at the point 
$e^{i\theta}$ there is a number $\beta(S)$ such that if $f$ 
is analytic in $\mathbb{D}$ then
\[
\int_0^1 \left({|f'(re^{i\theta})|\over 
\lambda_\mathbb{D}(re^{i\theta})}\right)^2\,
\lambda_\mathbb{D}(re^{i\theta})\,dr
\ \leqslant\ \beta(S)\, {\rm area}\Big(f(S)\Big).
\]
\end{theorem D*}

Note that this implies that, when $f(S)$ has finite area, then
$|f'(z)|/\lambda_\mathbb{D}(z)$ is square integrable with respect to hyperbolic
length on the hyperbolic ray $[0,1)$. Our proof of Theorem D is essentially
that given in \cite{MZ} but with the considerable benefits gained by using
hyperbolic geometry.

First, we clarify what we mean by a Stolz region. It is 
customary to consider a Stolz region at the point $1$ 
to be a Jordan domain whose closure lies in $\mathbb{D}\cup\{1\}$,
and whose boundary contains two Euclidean straight line 
segments which have a common endpoint at $1$. This is a 
Euclidean Stolz region and it is not conformally invariant. 
For most purposes it is better to use a hyperbolic Stolz 
region which we now describe. Let $\gamma_0$ be any geodesic 
ray in $\mathbb{D}$ ending at, say $\zeta$, and let $\gamma$ be
the entire geodesic that contains $\gamma_0$. For each positive
number $d$ the {\it lens region} $L(\gamma,d)$ is the set of 
points in $\mathbb{D}$ whose hyperbolic distance from $\gamma$ is 
strictly less than $d$; this is clearly the union of all 
open discs of radius $d$ whose centres lie in $\gamma$. 
A {\it hyperbolic Stolz region} at $\zeta$ is one 'end' of 
a lens region; for example, the union of all open discs of 
radius $d$ whose centres lie in the half-ray $\gamma_0$. 
If $\gamma$ is the real diameter $(-1,1)$ of $\mathbb{D}$, 
then $L(\gamma,d)$ is bounded by two arcs of circles, each 
passing through $\pm 1$, and it follows from this that in
almost every case in complex analysis one can use either a
Euclidean or a hyperbolic Stolz region. The advantage of
hyperbolic Stolz regions (and lens regions) is that they 
are defined in any domain that supports a hyperbolic metric, 
and they are conformally invariant. 

We are now ready to give our version of the proof of Theorem 
D and in this we shall use a hyperbolic Stolz region. 
We begin by stating two standard results in hyperbolic 
geometry (which, for completeness, we prove in the Appendix).
The first result is the hyperbolic version of the well known 
fact that if $f$ is analytic in a domain containing a closed 
Euclidean disc $\Delta$ with centre $z_0$, then $f(z_0)$ is 
the average value of $f$ over the disc $\Delta$. It is perhaps 
not so well known that if $f$ is analytic in $\mathbb{D}$,
the same is true when $\Delta$ is a closed hyperbolic disc.
The second result is a type of Harnack inequality 
for the hyperbolic density in $\mathbb{D}$.

\begin{lemma}
Suppose that $f$ is analytic in $\mathbb{D}$, and that $\Delta$ 
is a closed hyperbolic disc in $\mathbb{D}$ with hyperbolic centre 
$z_0$. Then $f(z_0)$ is the non-Euclidean average of $f$
taken over $\Delta$.
\end{lemma}

\begin{lemma}
For all $z$ and $w$ in $\mathbb{D}$, 
\[
{1\over 4\exp \rho_\mathbb{D}(z,w)} 
\ \leqslant\ {\lambda_\mathbb{D} (z)\over \lambda_\mathbb{D} (w)}
\ \leqslant\ 4\exp \rho_\mathbb{D}(z,w).
\] 
\end{lemma}

\begin{proof}[of Theorem D]
We may assume that $e^{i\theta}=1$. Let $S$ be 
the hyperbolic Stolz region consisting of points whose 
hyperbolic distance from $(0,1)$ is strictly less than $d$.
For any point $x$ on the diameter $(-1,1)$, 
let $Q(x)$ be the open hyperbolic disc with centre $x$ and 
radius $d$, and let the hyperbolic area of $Q(x)$ (which is 
independent of $x$) be $A(d)$. Then, using Lemma 6.1 followed
by the Cauchy-Schwarz inequality, we have
\begin{eqnarray*}
|f'(x)| &\leqslant & {1\over A(d)}\int\int_{Q(x)}|f'(w)|
\lambda_\mathbb{D} (w)^2\, dudv\\
&~\leqslant~& {1\over A(d)}
\sqrt{\int\int_{Q(x)}|f'(w)|^2\lambda_\mathbb{D} (w)^2\, dudv}
\sqrt{\int\int_{Q(x)}\lambda_\mathbb{D} (w)^2\, dudv}\\
&~=~&
{1\over \sqrt{A(d)}}
\sqrt{\int\int_{Q(x)}|f'(w)|^2\lambda_\mathbb{D} (w)^2\, dudv}.\\
\end{eqnarray*}
We deduce that
\begin{eqnarray}
{|f'(x)|^2\over \lambda_\mathbb{D} (x)^2}
&~\leqslant~& {1\over A(d)}
\int\int_{Q(x)}|f'(w)|^2
\left({\lambda_\mathbb{D} (w)\over \lambda_\mathbb{D} (x)}\right)^2
\, dudv\nonumber \\
&~\leqslant~& {4e^d\over A(d)} \int\int_{Q(x)}|f'(w)|^2\, dudv,
\end{eqnarray}
the last step being by Lemma 6.2.

Now define a sequence $x_n$ by $x_0=0$ and 
$x_0<x_1< \cdots < 1$, where $\rho(x_n, x_{n+1}) = d$. 
Next, let $y_n$ be the point in
$[x_n,x_{n+1}]$ that maximises 
$\big(|f'(x)|/\lambda_\mathbb{D} (x)\big)^2$. 
Then, using (6.2) with $x=y_n$ we obtain
\begin{eqnarray*}
\int_0^1 
\left({|f'(t)|\over \lambda_\mathbb{D} (t)}\right)^2\, 
\lambda_\mathbb{D}(t)dt
&~=~& 
\sum_{n=0}^\infty \int_{x_n}^{x_{n+1}} 
\left({|f'(t)|\over \lambda_\mathbb{D} (t)}\right)^2
\, \lambda_\mathbb{D}(t)dt\\
&~\leqslant~& 
d\sum_{n=0}^\infty \left({|f'(y_n)|\over 
\lambda_\mathbb{D} (y_n)}\right)^2 \\
&~\leqslant~&
{4de^d\over A(d)} \sum_{n=0}^\infty 
\int\int_{Q(y_n)}|f'(w)|^2\, dudv.\\
\end{eqnarray*}
Now each $Q(y_n)$ has radius $d$, and as
\[
\rho(y_n,y_{n+k}) \ \geqslant\  \rho(x_{n+1},x_{n+k}) \ \geqslant\ (k-1)d,
\]
we see that any point in $S$ can lie in at most five of 
the $Q(y_n)$. This implies that
\begin{eqnarray*}
\sum_{n=0}^\infty \int\int_{Q(y_n)}|f'(w)|^2\, dudv
&~\leqslant~&
5\int\int_{\cup_n Q(y_n)}|f'(w)|^2\, dudv\\
&~\leqslant~& 
5\int\int_S|f'(w)|^2\, dudv
\ \leqslant\ 
5\,{\rm area}(f(S)),\\
\end{eqnarray*}
and this completes our proof of Theorem D. Notice that this 
proof also gives an explicit estimate for the size of 
$\beta (S)$, namely
\[
\beta(S) \ \leqslant\  {20de^d\over A(d)} \ =\ 
{5de^d\over \pi\sinh^2 (\frac{1}{2} d)} \ \sim\ 20d/\pi
\]
as $d\to +\infty$.
\end{proof}

In order to apply Theorem D we need 
to know when the image of a Stolz region has finite area, 
and for this we quote the following result 
(Theorem 4 \cite[p.\,480]{MZ}). 

\begin{theorem E*}
Let $f$ be analytic in $\mathbb{D}$, and let $E$ be the set 
of points on $\partial\mathbb{D}$ where $f$ has an angular 
limit (from within a Stolz region). Then, for almost all 
$\zeta$ in $E$, and any Stolz region $S$ at $\zeta$, 
we have ${\rm area}\big(f(S)\big) < +\infty$.
\end{theorem E*}

\section{Refinements of Theorem A}

Although the exponent $1/2$ in Theorem A is the best possible, 
we can obtain more detailed information on the rate of growth
of ${\cal E}(r,\theta)$ by using integrals. Kennedy and Twomey 
\cite[Theorem 1]{KT} showed that if $A(f)$ is finite then 
\[ 
\int_0^1 {{{\cal E}(r,\theta)^2}\over
{(1-r)[\log 1/(1-r)]^2}}\,dr \ \leqslant\ {A(f)\over\pi} 
\]
for all $\theta$. As usual, it is sufficient to consider the 
case $\theta = 0$ and we write ${\cal E}(r)$ for ${\cal E}(r,0)$.
As $\ell(r) = \log (1+r)/(1-r)$ this inequality can be rewritten 
as 
\[
\int_0^1 {{{\cal E}(r)^2}\over{\ell(r)^2}}\ell'(r)\,dr
\ \leqslant\  C A(f),
\]
or better still as
\begin{equation}
\int_0^\infty {{{\cal E}(r)^2}\over{\ell(r)^2}}\,d\ell(r) 
\ \leqslant\ C A(f),
\end{equation}
for some constant $C$. We shall now state and prove an 
analogous result for $L^p\!$-spaces.

Let $f:\mathbb{D} \to \mathbb{C}$ be analytic and recall that
$|f'(z)|/\lambda_\mathbb{D}(z)$ is the factor by which $f$ 
changes scale from the hyperbolic metric at $z\in \mathbb{D}$ 
to the Euclidean metric in $\mathbb{C}$. The hyperbolic area 
measure on $\mathbb{D}$ is $\lambda_\mathbb{D}(z)^2\,dx\,dy$.
Let $\Lambda^p(\mathbb{D})$ be the set of analytic functions 
$f:\mathbb{D}\to \mathbb{C}$ for which
\[
||f'/\lambda_\mathbb{D}||_p \ =\ \left(\int_\mathbb{D} 
\left({|f'(z)|\over \lambda_\mathbb{D}(z)}\right)^p
\lambda_\mathbb{D}(z)^2\,dx\,dy\right)^{1/p}
\]
is finite. 
We will assume that $1<p<\infty$ and write $q$ for the conjugate 
index so that $1/p + 1/q = 1$. Note that, for $p=2$, we have 
\[
\Big(||f'/\lambda_\mathbb{D}||_2\Big)^2
\ =\ \int\!\!\int_\mathbb{D}|f'(z)|^2\,dx\,dy \ =\ A(f);
\]
thus $\Lambda^2(\mathbb{D})$ consists of those analytic functions 
with finite area $A(f)$ and so is the Dirichlet class
${\cal D}$. We now generalize (7.1).

\begin{theorem}
For a function $f$ in $\Lambda^p(\mathbb{D})$,
\[
\left(\int_0^\infty 
\left({{\cal E}(r,\theta)\over{\ell(r)}}\right)^p
\,d\ell(r)\right)^{1/p} 
\ \leqslant\ {{2q}\over{\pi^{1/p}}} ||f'/\lambda_\mathbb{D}||_p.
\]
Consequently,
\[
\int_0^1 {{{\cal E}(r,\theta)^p}\over
{(1-r^2)\left(\log{{1+r}\over{1-r}}\right)^p}}\,dr
\ <\ \infty.
\]
\end{theorem}

\begin{proof}
Poisson's formula shows that if $|z|<s<1$, then
\[
f'(z) \ =\ \int_0^{2\pi} f'(se^{i\theta})\, 
{{s^2 - |z|^2}\over{|z-se^{i\theta}|^2}}\,
{{d\theta}\over{2\pi}}.
\]
As $p>1$ we may apply Jensen's inequality to this and obtain
\[
|f'(z)|^p \ \leqslant\  
\int_0^{2\pi} |f'(se^{i\theta})|^p\; 
\left({{s^2 - |z|^2}\over{|z-se^{i\theta}|^2}}\,
{{d\theta}\over{2\pi}}\right)\\
\ \leqslant\ 
\left({{s+|z|}\over{s-|z|}}\right)
\int_0^{2\pi} |f'(se^{i\theta})|^p\,{{d\theta}\over{2\pi}}.
\]

Now let $z=r$ and $s = \sqrt{r}$; thus $0\leqslant r\leqslant s <1$. Then
\begin{eqnarray*}
\int_0^1 {{|f'(r)|^p}\over{\ell'(r)^p}}\,\ell'(r)\, dr 
&~\leqslant~&
\int_0^1 \left(\int_0^{2\pi} |f'(se^{i\theta})|^p \,
{{d\theta}\over{2\pi}}\right)\;
\left({{1-r^2}\over{2}}\right)^{p-1} 
\left({{s+r}\over{s-r}}\right)\, dr \\
&~=~& 
\int_0^1 \int_0^{2\pi} 
\left({|f'(se^{i\theta})|\over \lambda_\mathbb{D}(se^{i\theta})}
\right)^p \;
\left({{2}\over{1-s^2}}\right)^p \;
\left({{1-r^2}\over{2}}\right)^{p-1} \left({{s+r}\over{s-r}}\right)
\; {{1}\over{2\pi}}\, d\theta\, dr \\
&~=~& 
\int_0^1 \int_0^{2\pi} 
\left({|f'(se^{i\theta})|\over \lambda_\mathbb{D}(se^{i\theta})}
\right)^p \;
{{2(1+s^2)^{p-1} (1+s)^2}\over{\pi (1-s^2)^2}} \;
s\, d\theta\, ds \\
&~\leqslant~& 
\int_0^1 \int_0^{2\pi} 
\left({|f'(se^{i\theta})|\over \lambda_\mathbb{D}(se^{i\theta})}
\right)^p \;
{2^{p+2}\over\pi (1-s^2)^2}\,s\, d\theta\, ds \\ 
&~=~& 
{{2^p}\over{\pi}}\int\int_\mathbb{D} 
\left({|f'(z)|\over \lambda_\mathbb{D}(z)}
\right)^p \; \lambda_\mathbb{D}(z)^2\, dx\, dy.\\
\end{eqnarray*}
Therefore,
\begin{equation}
\int_0^\infty \left({{d{\cal E}}\over{d\ell}}\right)^p\,d\ell \ =\ 
\int_0^1 \left({{{\cal E}'(r)}\over{\ell'(r)}}\right)^p\,\ell'(r)\,dr
\ \leqslant\ {{2^p}\over{\pi}} 
(||f'/\lambda_\mathbb{D}||_p)^p 
\end{equation}

Now consider the integral 
\[
\int_0^L \left({{\cal E}(r)\over{\ell}(r)}\right)^p\,d\ell(r).
\]
Integration by parts gives
\begin{eqnarray*}
\int_0^L \left({{\cal E}\over{\ell}}\right)^p\,d\ell 
&~=~&
-{1\over{p-1}} \left.{{{\cal E}^p}\over{\ell^{p-1}}}\right|_0^L +
\int_0^L \left({{p}\over{p-1}}\right)\,
{{{\cal E}^{p-1}}\over{\ell^{p-1}}}\;
{{d{\cal E}}\over{d\ell}}\,d\ell \\
&~=~& 
-{1\over{p-1}} {{{\cal E}(L)^p}\over{L^{p-1}}}\; +\,
q \int_0^L \left({{{\cal E}}\over{\ell}}\right)^{p-1}\,
{{d{\cal E}}\over{d\ell}}\,d\ell \\
\end{eqnarray*}
So H\"older's inequality gives
\[
\left(\int_0^L \left({{\cal E}\over{\ell}}\right)^p \,
d\ell\right)^{1/p}
\ \leqslant\ q\left(\int_0^L \left({{d{\cal E}}\over{d\ell}}\right)^p \;
d\ell\right)^{1/p} 
\]
This, together with inequality (7.2), shows that
\[
\int_0^\infty \left({{\cal E}\over{\ell}}\right)^p\,d\ell
\ \leqslant\  {{2^p q^p}\over{\pi}} ||\sigma||_p^p 
\]
as required. This completes the proof of Theorem 7.1.
\end{proof}

The inequalities in Theorem 7.1 can be shown to be optimal 
by using the idea in Example 3.2, and we give a brief informal 
discussion of this. First, the following argument shows that
the integral inequality of Kennedy and Twomey is stronger than 
Theorem A. When the integral 
\[
\int_0^\infty \left({{\cal E}\over{\ell}}\right)^p\,d\ell
\]
converges we see that, for any $\delta > 0$, there is an $L$ with
\[ 
\int_L^\infty \left({{\cal E}\over{\ell}}\right)^p\,d\ell
\ \leqslant\ \delta.
\]
The distance ${\cal E}$ is an increasing function of $\ell$, so 
\[ 
{\cal E}(L)^p \int_L^\infty {1\over{\ell^p}}\,d\ell \ =\  
{{{\cal E}(L)^p}\over{(p-1)L^{p-1}}} \ \leq\ 
\int_L^\infty \left({{\cal E}\over{\ell}}\right)^p\,d\ell
\ \leqslant\  \delta 
\]
and hence ${\cal E} = o(\ell^{1/q})$.

Next, we consider the inequalities in Theorem 7.1.
Consider a continuous function $a:[0,\infty) \to (0,1]$ 
and the corresponding simply connected domain
\[
D(a) \ =\  \mathbb{D}\cup \{x+iy\in \mathbb{C}\ :\ x>0,\ |y|<a(x)\}.
\]
There is a conformal map $f:\mathbb{D} \to D(a)$ from 
$\mathbb{D}$ onto $D(a)$ with $f(0)=0$ and $f'(0) > 0$. 
This $f$ maps the radius $[0,1)$ 
onto the positive real axis in $D(a)$. Provided that $a$ varies 
slowly, say $|a'(x)| \leqslant 1$, it is easy to estimate the 
hyperbolic density $\lambda_{D(a)}(z)$ at any point $z$. 
If $t>0$, the domain $D(a)$ lies between the 
square $\{x+iy:|y| < a(t)-|x-t|\}$
and the slit plane $\mathbb{C} \setminus \{t+iy:|y| \geqslant a(t)\}$,
so the density lies between the densities for these two domains.
Hence, $\lambda_{D(a)}(x+iy) \sim 1/[a(x)-y)]$. This is
enough to show that $\sigma \in \Lambda^p(\mathbb{D})$ if and 
only if $a\in L^{p-1}[0,\infty)$.

In this situation, 
\[
\ell\ =\ \int_0^{\cal E} \lambda_{D(a)}(t)\,dt 
\qquad\hbox{ so }\qquad
{{d{\cal E}}\over{d\ell}} \ =\  {1\over{\lambda_{D(a)}}} \sim a.
\]
Thus the integral in (7.2) is
\[
\int_0^\infty \left({{d{\cal E}}\over{d\ell}}\right)^p\,d\ell 
\ \sim\  \int_0^\infty a(t)^p\,\lambda_{D(a)}(t)\,dt 
\ \sim\  
\int_0^\infty a(t)^{p-1}\,dt.
\]
It is now clear that (7.2) is optimal.

For the inequality in Theorem 7.1 we need to do a little more. 
The same argument gives
\[
\int_0^\infty \left({{\cal E}\over \ell}\right)^p\,d\ell 
\ \sim\ 
\int_0^\infty \left( {1\over{{\cal E}}}\int_0^{\cal E}
{1\over{a(t)}}\;dt \right)^{-p}\,dt\ . 
\]
Provided we look at functions $a$ that decay very slowly (so they are
essentially constant over longer and longer intervals as $t\to
\infty$), we can ensure that 
\[
{1\over{{\cal E}}}\int_0^{\cal E} {1\over{a(t)}}\;dt 
\ \sim\ 
{1\over{a({\cal E})}}.
\]
This shows that the inequality in Theorem 7.1 is optimal.

\section{Estimates on the change of scale}

Finally, we remark that inequalities between $|f'(z)|$ and 
$\lambda_\mathbb{D} (z)$ were studied extensively in \cite{SW} and 
although the results in \cite{SW} are not expressed in hyperbolic 
terms, we shall nevertheless write them in this form here.
There is an example \cite[p.\,151]{SW} which shows that
there exists a bounded analytic $f$ in $\mathbb{D}$, and a
sequence $z_n$ in $\mathbb{D}$, such that $|z_n|\to 1$ and
\begin{equation}
\liminf_{n\to\infty}\,{|f'(z_n)|\over\lambda_\mathbb{D} (z_n)} 
\ >\  0,
\end{equation}
and there is a proof (on p.146) that (8.1)
is best possible in the following sense. Let $Q(r)$ be 
defined and positive for $0<r<1$, and suppose that $Q(r)\to 0$ 
as $r\to 1$. Then there exists a function $f$, analytic and 
univalent in $\mathbb{D}$, and continuous in $\overline{\mathbb{D} }$, and 
a sequence $x_n$ satisfying $0<x_1<x_2< \cdots <x_n \to 1$, 
such that as $n\to \infty$,
\[
{|f'(x_n)|\over \lambda_\mathbb{D} (x_n)} \ =\  o\Big(Q(x_n)\Big).
\]

Next, \cite[Theorem 6, p.\,141, Corollary 8, p.\,143]{SW} , and
\cite[p.\,188]{Tsuji} contain the following result (which may be considered 
as a local version of Beurling's theorem although the 
exceptional set here has measure zero).

\begin{theorem F*}
Suppose that $f$ is analytic and univalent in $\mathbb{D}$.
Then for almost all $\theta$,
\[
f'(z) \ =\  o\left(\sqrt{\lambda_\mathbb{D} (z)}\right)
\]
as $z\to e^{i\theta}$ uniformly in any Stolz region at
$e^{i\theta}$. In particular,
for almost all points $\zeta$ on $\partial \mathbb{D}$, the
curve $f\big([0,\zeta]\big)$ is rectifiable.
\end{theorem F*}

A slightly sharper result is also given in \cite[Lemma 4, p.\,143]{SW}.

\begin{theorem G*}
Suppose that $f$ is analytic and univalent in $\mathbb{D}$.
Then for almost all $\theta$, as $r\to 1$,
\[
\int_r^1|f'(te^{i\theta})|\, dt \ =\ 
o\left(\sqrt{\lambda_\mathbb{D} (z)}\right).
\]
\end{theorem G*}

\section{Covering maps}

If the complement of a domain $D$ in $\mathbb{C}_{\infty}$
has at least three points then $\mathbb{D}$ is the universal 
covering space of $D$ and the hyperbolic metric on $\mathbb{D}$ 
projects under any universal covering map $\pi:\mathbb{D}\to D$ 
to the hyperbolic metric $\rho_D$ of $D$ which is defined by 
the line element $\lambda_D(w)|dw|$, where
\begin{equation}
\lambda_D (\pi (z))|\pi'(z)| = \lambda_\mathbb{D}(z).
\end{equation}
It follows from this that $\pi$ is a local isometry relative 
to the two hyperbolic metrics. When $D$ is multiply connected 
$\pi$ is not injective and so it is not a global isometry. 

Let us illustrate these ideas in the particular case 
when $D$ is an annulus.  Let 
\[
q(z) = i\left({1+z\over 1-z}\right), \qquad g(z) = \log (-iz),
\qquad h(z) = \exp (iz).
\]
Then $q$ maps $\mathbb{D}$ conformally onto the upper half-plane 
$\mathbb{H}$, $g$ maps $\mathbb{H}$ conformally onto the strip 
$S= \{x+iy : |y| < \pi /2\}$, and $h$ maps $S$ (conformally,
but not bijectively) onto the annulus
\[
A \ =\  \{ z : \exp(-\pi/2) < |z| < \exp (\pi /2)\}.
\]
We write $f = h\circ g \circ q$; then $f$ is a universal cover 
map of $A$, and the hyperbolic metric of $A$ has density
$\lambda_A$, where
\[
\lambda_A\big(f(z)\big)|f'(z)|
\ =\ \lambda_\mathbb{D}(z)
\ =\  {2\over 1-|z|^2}.
\]
A straightforward calculation shows that $z\in (-1,1)$ if and 
only if $|f(z)|=1$, and that for $z\in (0,1)$, we have
$\lambda_A\big(f(z)\big) = 1$. We deduce that in this example,
\begin{equation}
{\cal E}(r) \ =\  \int_0^r|f'(t)|\,dt
\ =\  \int_0^r\lambda_A\big(f(t)\big)|f'(t)|\,dt
\ =\  \int_0^r\lambda_\mathbb{D} (t)\,dt
\ =\  \ell (r).
\end{equation}

We would like to know whether or not there is a 
{\it geometric} explanation for the fact that we have an 
exponent $1/2$ in Theorem A but not in (9.2).
There is an obvious and significant geometric difference
between the case of a univalent map $f$ of $\mathbb{D}$ onto some
domain $D$, and the case of a covering map, say
$f_1$, of the annulus: in the former case the $f\!$-images of 
geodesics in $\mathbb{D}$ accumulate only on
the boundary of $D$, whereas in the case of the annulus 
described above, we have constructed a geodesic in $\mathbb{D}$ whose 
image under the covering map does not accumulate anywhere on 
the boundary of the annulus. It is reasonable, then, to ask 
whether this geometric difference accounts for the different 
exponents in (3.1) and (9.2).
We note in passing that some type of boundedness 
condition on $f$ is necessary in order to ensure that 
${\cal E}(r)$ grows no faster than some power of $\ell (r)$.
Indeed, for the function $f(z)=i(1+z)/(1-z)$, which is a 
conformal map of $\mathbb{D}$ onto the upper half-plane $\mathbb{H}$, 
we have ${\cal E}(r) \sim \exp[\ell(r)]$ as $r\to 1$.

We shall restrict ourselves here to a description of a simple 
case of covering maps and this will be sufficient to substantiate
the ideas described above in many cases. The following
discussion is simply a generalization of the case of the
annulus described above.

Let $D$ be a bounded domain of finite connectivity bounded 
by a finite number, say $C_1,\ldots ,C_k$ of (sufficiently) 
smooth closed curves (none of which degenerate to a point).
In such circumstances there is a universal covering map 
$f:\mathbb{D}\to D$ which is invariant under a Fuchsian group $G$ 
acting on $\mathbb{D}$ and which is such that $D$ is conformally 
equivalent to the quotient surface $\mathbb{D} /G$.
These assumptions on $D$ imply the following facts; here, we 
assume familiarity with the basic theory of Fuchsian groups.

\begin{enumerate}
\item[(1)]
The group $G$ is finitely generated, and its limit set $L$ 
on $\partial \mathbb{D}$ has Hausdorff dimension strictly between 
$0$ and $1$. In particular, $L$ has positive capacity and 
zero linear measure.

\item[(2)]
Each ray from $0$ in $\mathbb{D}$ that ends at a point $e^{i\theta}$ in
$\partial\mathbb{D} \backslash L$ has a rectifiable $f\!$-image
that ends at a point on $\partial \mathbb{D}$. For such rays, 
${\cal E}(r,\theta) = O(1)$ as $r\to 1$.

\item[(3)]
Each ray from $0$ in $\mathbb{D}$ that ends at a point $e^{i\theta}$ in
$L$ has an $f\!$-image that returns infinitely often to some
compact subset of $\mathbb{D}$.
\end{enumerate}

We shall now obtain estimates of ${\cal E}(r,\theta)$ in
this situation. Roughly speaking, our result show that the
calculations given above for the annulus are typical of this 
situation.

\begin{theorem}
Let $f$ be a universal cover map of a bounded domain $D$ as 
described above. Then, for every $\theta$,
\begin{equation}
{\cal E}(r,\theta) = O\big(\ell(r)\big)
\end{equation}
as $r\to 1$. Moreover, 
\begin{enumerate}
\item[(i)]
if $e^{i\theta} \notin L$ then
${\cal E}(r,\theta) = O(1)$ as $r\to 1$, and
\item[(ii)]
if $e^{i\theta}$ is fixed by some hyperbolic element in $G$
(and such points are dense in $L$), then there is some positive
number $K$ (that depends on $\theta$) such that for all $r$,
\begin{equation}
{\cal E}(r,\theta) \geqslant K\ell(r).
\end{equation}
\end{enumerate}
\end{theorem}

We suggest that the way that ${\cal E}(r,\theta)$ varies in 
comparison with $\ell(r)$ might be governed entirely by 
nature of the limit point $e^{i\theta}$ as exhibited by the 
geometric behaviour of the $f\!$-image of the geodesic ray, 
say $\gamma$, from $0$ to $e^{i\theta}$. Explicitly, we make 
the following 

\vfill\break
\begin{conjecture*}
\begin{enumerate}
\item[(i)]
If $e^{i\theta}$ is not in $L$, then $f(\gamma)$ ends at a 
single point on the boundary of $D$ (or, more generally, 
if $\partial D$ is not smooth, at a prime end on $\partial D$;
see Example 3.3), and in this case we have
${\cal E}(r,\theta) = O(\ell(r)^{1/2})$. 
\item[(ii)]
If $e^{i\theta} \in L$, then $f(\gamma)$ returns infinitely 
often to a compact subset of $D$ (this is a known fact), and 
this probably implies that ${\cal E}(r,\theta)$ grows rather 
rapidly with $r$. In this case we might expect that
${\cal E}(r,\theta) = O(\ell(r)^{\beta})$, where $\beta$
is in $(1/2,1]$ and depends on the geometric position of 
$f(\gamma)$. 
\item[(iii)]
In the extreme case when $e^{i\theta}$
is a hyperbolic fixed point, $f(\gamma)$ lies entirely 
within a compact subset $K$, say, of $D$. In this case, 
the Euclidean and hyperbolic lengths along $f(\gamma)$ are 
comparable, and we find that ${\cal E}(r,\theta)/\ell(r)$
is bounded above and below by positive finite constants
as is given by (9.3) and (9.4).
This is the case of the annulus discussed above.
\end{enumerate}
\end{conjecture*}

\begin{proof}[of Theorem 9.1] 
As $D$ is bounded it lies in some disc $\{|z|<M\}$.
If we now apply the Schwarz-Pick Lemma to the function $f/M$ 
(which maps $\mathbb{D}$ into itself) we obtain
\[
{2|f'(z)|/M \over 1-(|f(z)|/M)^2} \ \leqslant\  \lambda_\mathbb{D} (z).
\]
This gives $|f'(z)| \leqslant (M/2)\lambda_\mathbb{D}(z)$,
and integration of both sides gives a sharper form of (9.3),
namely
${\cal E}(r,\theta) \leqslant (M/2)\ell(r)$.

To establish (9.4) suppose that $e^{i\theta}$ is a fixed point 
of some hyperbolic element in $G$; then the $f\!$-image of the
ray to $e^{i\theta}$ lies in a compact subset of $D$ so that
$\lambda_D$ is bounded above, say by a positive number $k$, on 
this image. As the covering map is a local isometry, we find 
from (9.1) that
\[
{\cal E}(r,\theta) 
\ =\  \int_0^r |f'(te^{i\theta})| \,dt
\ \geqslant\  {1\over k}\int_0^r |f'(te^{i\theta})| 
\lambda_D\big(f(te^{i\theta})\big)\,dt 
\ =\  {\ell(r)\over k}.
\]
This completes the proof.
\end{proof}

Theorem A shows that in a simply connected domain with finite 
area we have ${\cal E} = o(\ell^{1/2})$. However, we have seen 
that ${\cal E}$ can grow like the first power $\ell$ when we
measure lengths along a circular path in an annulus. It seems
likely that, given any $\alpha$ in $(0,1)$, one could construct 
a multiply connected domain with finite area in which 
${\cal E} \sim \ell^\alpha$ for a suitable geodesic.  We give 
a brief sketch of an idea which we believe will work; however, 
the estimates needed would be very delicate and we have not 
completed the details.

Consider a domain $D$ obtained by abutting a sequence of annuli $A_k$,
as in the diagram. Here $A_k$ is an annulus of inner radius
$\textstyle{1\over 2} r_k$ and outer radius $2r_k$. The total area is
approximately $\sum \textstyle{{15}\over{4}} \pi r_k^2$, so this sum
must converge. For any sequence of natural numbers $n_k$, there is a
hyperbolic geodesic $\gamma$ in $D$ that starts at a fixed point $z_o
\in A_1$ and winds $n_k + \textstyle{1\over 2}$ times around $A_k$
before moving to $A_{k+1}$. (There is such a geodesic since every
homotopy class contains a geodesic.) The Euclidean length of the
geodesic from $z_o$ until it leaves $A_N$ is asymptotically:
\[
{\cal E}_N \ \sim\  \sum_{k=1}^N (n_k + \textstyle{1\over 2}) 2\pi r_k\ .
\]

\begin{center}
\epsfig{file=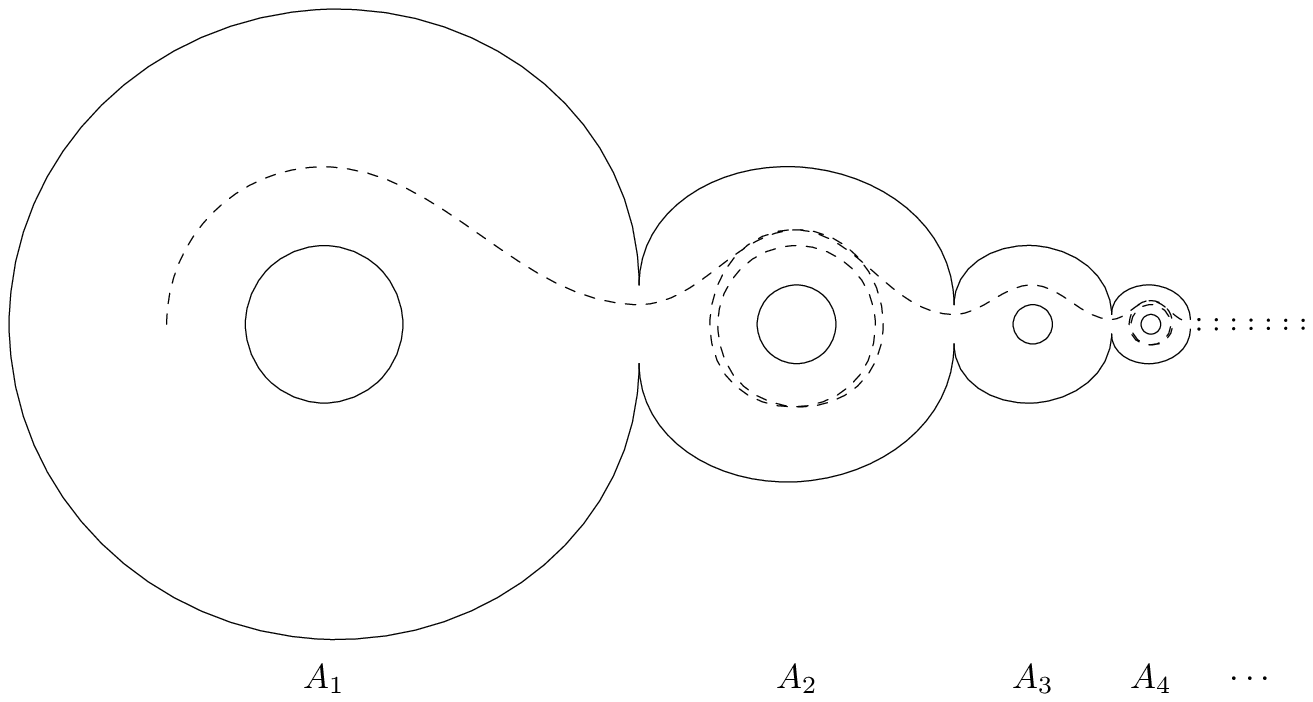}
\end{center}

\centerline{{\bf Figure 2.}\ (The geodesic is shown as a dashed line.)}
\bigskip

The part of $\gamma$ within $A_k$ lies at a distance of approximately
$\textstyle{1\over 2} r_k$ from the boundary, so its hyperbolic length
is roughly 
\[
(n_k + \textstyle{1\over 2}) 2\pi r_k {2\over{r_k}} \ =\  
(n_k + \textstyle{1\over 2}) 4\pi.
\]
Thus the hyperbolic length of the geodesic from $z_o$
until it leaves $A_N$ is asymptotically:
\[
\ell_N \ \sim\  \sum_{k=1}^N (n_k + \textstyle{1\over 2}) 4\pi.
\]

For any $\beta > 0$, set $r_k = 1/k$ and $n_k = k^\beta$. Then $\sum
r_k^2$ converges, so $D$ will have finite area. Moreover, 
\[
{\cal E}_N \ \sim\  \sum_{k=1}^N n_k r_k 
\ =\  \sum_{k=1}^N k^{\beta - 1} 
\ \sim\  N^\beta
\]
and 
\[
\ell_N \ \sim\  \sum_{k=1}^N n_k 
\ =\  \sum_{k=1}^N k^\beta 
\ \sim\  N^{\beta + 1}
\]
so ${\cal E}_N \sim \ell^{\beta/(\beta+1)}$.

\section{The hyperbolic length of the image curve} 

In this section we shall consider functions $f$ that are
analytic in $\mathbb{D}$, and such that the image domain $f(\mathbb{D})$ 
supports a hyperbolic metric, and we shall say that $f$ is 
{\it hyperbolic} in $\mathbb{D}$ whenever this is so.
Not every analytic map in $\mathbb{D}$ is hyperbolic;
for example, if $g$ maps $\mathbb{D}$ conformally onto the strip 
given by $0<{\rm Im}[z]<3\pi$, then $\exp\circ g$ is not 
hyperbolic. However, many functions analytic in $\mathbb{D}$ are
hyperbolic; for example, all bounded functions, and all 
univalent functions, are.

Suppose now that $f$ is analytic in $\mathbb{D}$ and 
hyperbolic. Then, instead of comparing
${\cal E}(r)$ with the hyperbolic length $\ell(r)$ of 
$[0,r]$, we can attempt to 
compare it with the hyperbolic length of the image curve
$f([0,r])$. Let us denote this length by ${\cal H}(r)$;
thus ${\cal E}(r)$ and ${\cal H}(r)$ denote the Euclidean
and hyperbolic lengths, respectively, of $f([0,r])$, and
we are asking for a comparison of the Euclidean and hyperbolic 
lengths of the curve $f([0,r])$.
We emphasize that the problem of estimating ${\cal E}(r)$ 
in terms of ${\cal H}(r)$ is more difficult than the 
original problem of obtaining an estimate in terms of 
$\ell (r)$ because the Schwarz-Pick Lemma shows that for 
any analytic hyperbolic $f$ we have ${\cal H}(r) \leqslant \ell(r)$.
Of course, if $f$ is univalent in $\mathbb{D}$ (or if $f$ is a 
covering map of $\mathbb{D}$ onto $f(\mathbb{D} )$), then ${\cal H}(r)=
\ell (r)$ so the two problems are the same.
We shall now prove the following result.

\begin{theorem}
Suppose that $f$ is analytic and bounded in $\mathbb{D}$, and that
$f$ is hyperbolic. Then as $r\to 1$. 
\begin{equation}
{\cal E}(r) =O\big({\cal H}(r)\big)
\end{equation}
\end{theorem}

\begin{proof}
We remark first that the example of the universal covering map of an annulus
(or indeed, any covering map) shows that one cannot do better than the
estimate (10.1) for in these cases, ${\cal H}(r)=\ell(r)$.  The proof follows
the same lines as the proof of Theorem A in \cite{Keogh}, except that we use
integrals instead of power series. Using the Cauchy-Schwarz inequality we
obtain
\begin{eqnarray*}
{\cal E}(r)^2
&~=~& \left(\int_0^r |f'(t)|\,dt\right)^2 \\
&~=~& \left(\int_0^r 
\sqrt{\lambda_D(f(t))|f'(t)|}\sqrt{|f'(t)|
\over\lambda_D (f(t))}\,dt\right)^2 \\
&~\leqslant~& \left(\int_0^r\lambda_D(f(t))|f'(t)|\,dt \right)
\left(\int_0^r{|f'(t)|\over \lambda_D (f(t))}\,dt\right)\\
&~=~& {\cal H}(r)
\left(\int_0^r{|f'(t)|\over \lambda_D (f(t))}\,dt\right).\\
\end{eqnarray*}
Now as $D$ is bounded, $\lambda_D$ has a positive lower bound,
say $\mu$, in $D$ and we obtain the inequality
${\cal E}(r)^2 \leqslant {\cal H}(r){\cal E}(r)/\mu$, or
${\cal E}(r) \leqslant {\cal H}(r)/\mu$.
This completes the proof of Theorem 10.1.
\end{proof}

Finally, we observe that the proof of Theorem 10.1 also provides
a proof of the following result.

\begin{theorem}
Suppose that $f$ is analytic and hyperbolic in $\mathbb{D}$, and 
that the function $\lambda_D$ is bounded below on the
ray $[0,e^{i\theta})$. Then {\rm (10.1)} holds as $r\to 1$. 
\end{theorem}

\section{Closing remarks} 

We end with the remark that in the general case of a
map $f$ analytic in $\mathbb{D}$, if $f$ is smooth enough on, say,
an arc of $\partial\mathbb{D}$ centred at $e^{i\theta}$, then we 
might expect that for a sufficiently small positive $\delta$,
the $f\!$-image of $\mathbb{D}\cap \{z:|z-e^{i\theta}|<\delta\}$
has finite area. If this is so then, from Theorem B,
we have
\begin{equation}
{\cal E}(r,\theta) = o\big(\sqrt{\ell(r)}\big)
\end{equation}
as $r\to 1$. Thus, in some general sense, if the 
$f\!$-image of a geodesic ray ends (smoothly) on
the boundary of $f(\mathbb{D} )$, or, more generally, on the
boundary of the Riemann surface for the function $f^{-1}$,
then we expect that (11.1) holds. By contrast, if the 
$f\!$-image of the geodesic ray circulates infinitely often 
around one of the `holes' in $f(\mathbb{D} )$, then 
${\cal E}(r,\theta)$ and $\ell(r)$ are of a comparable 
size as $r\to 1$.
It may be that more delicate estimates relate the comparative sizes
of ${\cal E}(r,\theta)$ and $\ell(r)$ to the finer detail
of the geometry of the images of the geodesic rays.

\oneappendix
\section{Appendix: Hyperbolic geometry}

We complete this paper by giving the proofs of Theorem 3.4, and of Lemmas 6.1
and 6.2.  We begin with a proof of Theorem 3.4, This will follow from a result
showing how hyperbolic geodesics move when we expand a domain and we will use
the observation that a geodesic is the level set for harmonic measure.  Let
$\alpha$ be a closed arc of the unit circle and $\beta$ the complementary arc.
Then the harmonic measure for $\alpha$ in ${\mathbb D}$ is the unique harmonic
function $\omega : {\mathbb D} \to (0, 1)$ with boundary values $0$ on $\beta$
and $1$ on the interior of $\alpha$.  There are no boundary values at the
endpoints $a_0, a_1$ of $\alpha$ but the harmonic measure is bounded there.
It is simple to calculate the harmonic measure explicitly and see that the
level set $\{ z\in {\mathbb D} : \omega(z) = \frac{1}{2} \}$ is the hyperbolic
geodesic from $a_0$ to $a_1$.

A similar result holds for domains with a Jordan boundary.  Let $E$ be a
simply-connected, hyperbolic, Jordan domain with $\alpha$ as a closed arc of
$\partial E$ and $\beta$ as the complementary arc of $\partial E$.  Then the
Riemann map $f : E \to {\mathbb D}$ extends continuously to the boundary
$\partial E$ and maps it homeomorphically onto the unit circle.  If $\omega$
is the harmonic measure of the arc $f(\alpha)$ in ${\mathbb D}$, then $\theta
= \omega \circ f$ is the harmonic measure of $\alpha$ in $E$.  This function
$\theta : E \to (0, 1)$ is harmonic with boundary values $0$ on $\beta$ and
$1$ on the interior of $\alpha$.  The hyperbolic geodesic between the
endpoints $a_0, a_1$ of $\alpha$ is then the level set $\{ z\in E : \theta(z)
= \frac{1}{2} \}$ since $f$ is a hyperbolic isometry.

\begin{proposition*}
Let $D, E$ be two simply-connected, hyperbolic domains that satisfy:
\begin{enumerate}
\item $E$ is a Jordan domain whose boundary $\partial E$ consists of a closed
  arc $\alpha$, joining the endpoints $a_0$ and $a_1$, and the complementary
  arc $\beta$;
\item $E$ is a subset of $D$ with $\alpha \subset \partial D$ and $\beta
  \subset D$.
\end{enumerate}
Then a geodesic $\delta$ from $a_0$ to $a_1$ in $D$ and the geodesic $\eta$
from $a_0$ to $a_1$ in $E$ do not meet in $E$.  Indeed, $\eta$ separates
$\delta$ and $\alpha$.
\end{proposition*}

\begin{proof}
We first show that it is sufficient to prove the result when $D$ is the unit
disc.  Let $f : D \to {\mathbb D}$ be a Riemann map for $D$.  Since $E$ is a
Jordan domain, each boundary point $\zeta$ of $\alpha$ is accessible.  Also,
since $E \subset D$, the point $\zeta$ is also an accessible boundary point of
$D$ defined by an approach to $\zeta$ within $E$.  Each such point $\zeta \in
\alpha$ corresponds to a unique prime end for $D$ defined by an approach to
$\zeta$ within $E$.  These prime ends fill out a closed arc
$\widetilde{\alpha}$ of $\partial{\mathbb D}$ with endpoints $\widetilde{a}_0,
\widetilde{a}_1$.  The image ${\mathbb E} = f(E)$ is a subdomain of the unit
disc ${\mathbb D}$ and is bounded by the closed arc $\widetilde{\alpha}$ and
the arc $\widetilde{\beta} = f(\beta)$ that lies in the interior of ${\mathbb
D}$.  The conformal map $f$ preserves geodesics, so it is sufficient to prove
the proposition when $D$ is the unit disc.

In this situation, both $D$ and $E$ are Jordan domains and so the harmonic
measures extend continuously to the boundary except at endpoints.  Let
$\theta : E \to (0, 1)$ be the harmonic measure of $\alpha$ in $E$ and let
$\omega : D \to (0, 1)$ be the harmonic measure of $\alpha$ in the unit disc
$D$.  Then $\theta$ has boundary values $0$ on $\beta$ and $1$ on the interior
of $\alpha$.  The restriction of $\omega$ to $E$ has boundary values
$\omega(z) > 0$ at each $z \in \beta$ and $1$ at each interior point of
$\alpha$.  Hence, 
\[
\omega(z) > \theta(z) \qquad \hbox{ for each } z \in E \ .
\]
In particular, the geodesics $\delta = \{z \in D : \omega(z) = \frac{1}{2} \}$
and $\eta = \{ z\in E : \theta(z) = \frac{1}{2} \}$ are disjoint.
Furthermore, $\theta(z) < \frac{1}{2}$ for each $z\in \delta$, so the
geodesic $\eta$ lies in the region $\{ z\in D : 1 > \omega(z) > \frac{1}{2} \}$
bounded by $\alpha$ and $\delta$. 

\end{proof}

We now complete the proof of Theorem 3.4.  It will suffice to show that
the geodesic $\gamma$ in $D$ does not cross the segments $[-1, -k] \times
\{0\}$ and $[k, 1] \times \{0\}$ for some constant $k$ with $0 < k < 1$.
Let $E$ be the square $(0, 1) \times (-\frac{1}{2}, \frac{1}{2})$,  Then $E$ is
a Jordan domain lying within $D$ and the line segment $\alpha = \{1\} \times
[-\frac{1}{2}, \frac{1}{2}]$ lies in the boundary of both $E$ and $D$.  Label
the endpoints of this line segment $a_0 = (1, -\frac{1}{2})$ and $a_1 = (1,
\frac{1}{2})$.  Let $\delta$ be the hyperbolic geodesic in $D$ joining $a_0$
to $a_1$ and let $\eta$ be the hyperbolic geodesic in $E$ joining $a_0$
to $a_1$. The reader is urged to draw a diagram.  

The arcs $\eta$ and $\alpha$ bound a Jordan region that intersects the
$x$-axis in some segment $[k, 1]$.  The last proposition shows that the
geodesic $\delta$ in $D$ can not meet this segment.  

By hypothesis, the geodesic $\gamma$ in $D$ does not have an endpoint in
$\alpha$.  By using the Riemann map from $D$ to the unit disc, we see that
$\gamma$ can not meet the region bounded by $\delta$ and $\alpha$, or it would
have at least one endpoint in $\alpha$. Therefore it certainly can not meet
the segment $[k, 1]$.  An entirely similar argument shows that $\gamma$ can
not meet $[-1, -k]$ and completes the proof of Theorem 3.4.

For further information about the hyperbolic geometry of
rectangles, see \cite{Beardon1} and \cite{Beardon2}.

Next, we give a formal statement and proof of Lemma 6.1.

\begin{lemma A*}
Suppose that $f$ is analytic in $\mathbb{D}$, and that $\Delta$ 
is a closed hyperbolic disc with hyperbolic centre $z_0$. Then
\begin{equation}
f(z_0) \ =\  {1\over A_h(\Delta)}
\int\int_{\Delta} f(z) \lambda_\mathbb{D}(z)^2\, dxdy,
\end{equation}
where $A_h(\Delta)$ is the hyperbolic area of $\Delta$.
\end{lemma A*}

\begin{proof}
First we prove this when $z_0=0$. In this case $\Delta$ is of the form
$\{z:|z|\leqslant R\}$, where $R<1$, and the result follows from the Euclidean
form of the mean value property.  Explicitly, we have
\[
\int\int_{\Delta} f(z) \lambda_\mathbb{D}(z)^2\, dx\,dy 
\ =\  \int_0^R\int_{\theta=0}^{2\pi}
{4f(re^{i\theta})\over (1-r^2)^2}\,r\,dr\,d\theta.
\]
We now integrate with respect to $\theta$ first, and using the fact that
$f(0)$ is the mean value of $f$ over the circle $|z|=r$, we obtain
\[
\int\int_{\Delta} f(z) \lambda_\mathbb{D}(z)^2 \,dx\,dy 
\ =\  8\pi f(0)\int_0^R
{1\over (1-r^2)^2}\,r\,dr\,d\theta.
\]
As this holds when $f$ is the constant function $1$, the integral on the right
must be $A_h(\Delta)/8\pi$ and this gives (A.1) when $z_0=0$.

Now suppose that $\Delta$ is any hyperbolic disc with centre $z_0$, say. There
is a M\"obius map $\gamma$ of $\mathbb{D}$ onto itself such that
$\gamma(\Delta) = \{z:|z|\leqslant R\}=\Delta'$, say, and $\gamma(z_0)=0$.
Now the function $F$ given by $F=f\circ \gamma^{-1}$ is analytic in $\Delta'$
and $F(0) = f(z_0)$ so that $f(z_0)$ is the average value of $F$ over the disc
$\Delta'$. As
\[
\lambda_\mathbb{D} \big(\gamma(z)\big) |\gamma'(z)| 
= \lambda_\mathbb{D} (z)
\]
(this is just (3.2) again), this shows that
\begin{eqnarray*}
A_h(\Delta)f(z_0) 
&~=~& \int\int_{\Delta'}F(z)\lambda_\mathbb{D} (w)^2\,dudv\\
&~=~& \int\int_{\Delta}F\gamma(z)
\lambda_\mathbb{D} \big(\gamma(z)\big)^2|\gamma'(z)|^2\,dxdy
\ =\  \int\int_{\Delta}f(z)
\lambda_\mathbb{D}(z)^2\,dxdy\\
\end{eqnarray*}
as required.
\end{proof}

Finally, we give the proof of Lemma 6.2.

\begin{proof}[of Lemma 6.2]
This is easy. Take $z$ and $w$ in $\mathbb{D}$, and without loss
of generality we may suppose that $|z| \leqslant |w|$. Then
\begin{eqnarray*}
\exp \rho (z,w) 
&~\geqslant~& \exp \rho (|z|,|w|) \\
&~\geqslant~& \exp \big[\rho (0,|w|)- \rho (0,|z|)\big] \\
&~=~& \exp \left[ \log\left({1+|w|\over 1-|w|}\right)
-\log\left({1+|z|\over 1-|z|}\right)\right]\\
&~=~& \left({1+|w|\over 1+|z|}\right)^2{1-|z|^2\over 1-|w|^2}\\
&~\geqslant~& {\lambda_\mathbb{D} (w) \over 4\lambda_\mathbb{D} (z)}.
\end{eqnarray*}
The other inequality follows by symmetry.
\end{proof}


\affiliationone{
   Department of Pure Mathematics and Mathematical Statistics\\
   Wilberforce Road\\
   Cambridge CB3 0WB\\
   Great Britain
   \email{afb@dpmms.cam.ac.uk\\
     tkc@dpmms.cam.ac.uk}}


\begin{thebibliography}{99}
%
%
\bibitem{Ahlfors}
{\bibname L.V. Ahlfors}, 
{\em Conformal invariants, topics in geometric function theory}
{McGraw-Hill, New York, 1973}.

\bibitem{ACP}
{\bibname J.M. Anderson, J. Clunie \and Ch. Pommerenke},
`On Bloch functions and normal functions',
{\em J.Reine Angew. Math.} 270 (1974), 12-37.

\bibitem{BB}
{\bibname Z. Balogh \and M. Bonk},
`Lengths of radii under conformal maps of the unit disc',
{\em Proc. Amer. Math. Soc.} 127 (1999), 801-804.

\bibitem{Beardon1} 
{\bibname A.F. Beardon}, 
`The hyperbolic metric of a rectangle',
{\em Ann. Acad. Sci. Fenn.} 26 (2001), 401-407.

\bibitem{Beardon2} 
{\bibname A.F. Beardon}, 
`The hyperbolic metric of a rectangle II',
to be published in {\em Ann. Acad. Sci. Fenn.}.

\bibitem{BP} 
{\bibname A.F. Beardon, \and Ch. Pommerenke},
`On the Poincar\'e metric of plane domains',
{\em J. London Math. Soc.} (2), 18 (1978), 475-483.

\bibitem{Beurling} 
{\bibname A. Beurling}, 
`Ensembles exceptionnels',
{\em Acta Math.} 72 (1940), 1-13.

\bibitem{CT} 
{\bibname T. Carroll \and J.B. Twomey}, 
`Conformal mappings of close-to-convex domains', 
{\em J. London Math. Soc.}, (2) 55 (1997), 489-498.

\bibitem{Jenkins} 
{\bibname J.A. Jenkins}, 
`On a result of Keogh',
{\em J. London Math. Soc.}, 31 (1956), 391-399.

\bibitem{Kennedy} 
{\bibname P.B. Kennedy},
`Conformal mapping of bounded domains',
{\em J. London Math. Soc.}, 31 (1956), 332-336.

\bibitem{KT} 
{\bibname P.B. Kennedy \and J.B. Twomey},
`Some properties of bounded univalent functions and 
related classes of functions',
{\em Proc. Royal Irish Acad. Sect.} A65 (1967), 43-49.

\bibitem{Keogh} 
{\bibname F.R. Keogh}, 
`A property of bounded schlicht functions',
{\em J. London Math. Soc.}, 29 (1954), 379-382.

\bibitem{MZ} 
{\bibname J. Marcinkiewicz \and A. Zygmund},
`A theorem of Lusin',
{\em Duke Math. J.}, 4 (1938), 473-485.

\bibitem{Pommerenke} 
{\bibname Ch. Pommerenke},
{\em Univalent functions},
(Vandenhoeck and Ruprecht, G\"ottingen, 1975).

\bibitem{SW} 
{\bibname W. Seidel \and J.L. Walsh}, 
`On the derivatives of functions analytic in the unit circle and their
radii of univalence and of $p$-valence',
{\em Trans. American Math. Soc.} 52 (1942), 128-216.

\bibitem{Solynin} 
{\bibname A. Yu. Solynin}, 
`Functional Inequalities via Polarization',
{\em Algebra i Analiz} 8 (1996), no. 6, 148-185; 
translation in  {\em Petersburg Math. J.} 8 (1997), no. 6, 1015-1038.

\bibitem{Tsuji} 
{\bibname M. Tsuji},
{\em Potential theory in modern function theory},
(Maruzen, Tokyo, 1959).

\end{thebibliography}
\end{document}